\documentclass[11pt,reqno]{amsart}
\usepackage{color, amsmath,amssymb, amsfonts, amstext,amsthm, latexsym}
\allowdisplaybreaks
\newcommand{\cF}{{\cal F}}
\newcommand{\RR}{{\mathbb R}}

\newcommand{\EX  }{{\mathbb E}}
\newcommand{\EE}{{\mathbb E}}

  \newcommand{\PX}{{\mathbb P}}
\newcommand{\PP}{{\mathbb P}}

\newcommand{\s}{\sigma}
\newcommand{\e}{\varepsilon}

\newcommand{\om}{\omega}
\newcommand{\Om}{\Omega}

\renewcommand{\cF}{\mathcal F}

\numberwithin{equation}{section}
\newtheorem{theorem}{Theorem}[section]
\newtheorem{definition}[theorem]{Definition}

\newtheorem{proposition}[theorem]{Proposition}

\newtheorem{Def}[theorem]{Definition}

\newtheorem{Prop}[theorem]{Proposition}

\begin{document}
\title[LDP and the zero viscosity limit for 2D NSE]{Large Deviations
and the zero viscosity limit for  2D stochastic Navier Stokes Equations with  free boundary}
\author[H. Bessaih]{Hakima Bessaih}
\address{University of Wyoming, Department of Mathematics, Dept. 3036, 1000
East University Avenue, Laramie WY 82071, United States}
\email{ bessaih@uwyo.edu}

\author[A. Millet]{ Annie Millet}
\address{SAMM, EA 4543,
Universit\'e Paris 1 Panth\'eon Sorbonne, 90 Rue de
Tolbiac, 75634 Paris Cedex France {\it and} Laboratoire de
Probabilit\'es et Mod\`eles Al\'eatoires,
  Universit\'es Paris~6-Paris~7, Bo\^{\i}te Courrier 188,
      4 place Jussieu, 75252 Paris Cedex 05, France}
\email{annie.millet@univ-paris1.fr {\it and}
annie.millet@upmc.fr}

\begin{abstract}
Using a weak convergence approach, we prove a LPD for the solution of  2D
stochastic Navier Stokes equations when the viscosity converges to 0 and
the noise intensity is multiplied by the square root of the viscosity.
Unlike previous results on LDP for hydrodynamical models, the weak convergence
is proven by tightness properties of the distribution of the solution in appropriate functional spaces.
\end{abstract}
\maketitle
 \maketitle


\section{Introduction}\label{s1}
The vanishing viscosity limit for solutions of Navier-Stokes equations is a singular limit,
that means that the type of the equation may change in the limit.
Singular limits are ubiquitous in applied mathematics and correspond to physical reality.
In bounded domains, the vanishing viscosity limit
shows a physical phenomena called boundary layers.
The Navier-Stokes equations are second-order differential equations and require Dirichlet boundary conditions,
while the Euler equation only  requires for the particle paths to be tangent to the boundary.
At present the problem of vanishing viscosity limit is open even in two dimensions in bounded domains,
while more progress has been made in the study of the limit
when there is no boundary
or when we impose particular boundary conditions, like the one we are considering in this paper.
There are two distinct concepts of vanishing viscosity limit.
The finite-time, zero viscosity limit of solutions of the Navier-Stokes with a fixed initial datum and with
time $t$ in some finite interval $[0,T]$. By contrast, in the infinite-time zero-viscosity limit,
long-time averages of functionals of the solutions are considered first at fixed $\nu$.
These are represented by measures $\mu^{\nu}$ in functions spaces and the zero-viscosity limit
$\lim_{\nu\rightarrow 0}\mu^{\nu}$ is then studied. The two kinds of limits are not the same.

In the present paper, we are dealing with flows described by stochastic Navier Stokes equations in dimension 2
of the following form
\begin{equation}\label{Euler}
\frac{\partial u_\nu(t)}{\partial t}- \nu \Delta u_\nu(t) + (u_\nu(t) \cdot\nabla)u_\nu(t) =-\nabla p+ G_\nu (t, u_\nu(t))\frac{\partial W(t,.)}{\partial t},
\end{equation}
in an open bounded domain $D$ of $\RR^{2}$ with a smooth boundary
$\partial D$ which satisfies the locally Lipschitz condition see
\cite{adams}. Here,  $u_\nu$ is the velocity of the fluid, $\nu>0$ is its
viscosity, $p$ denotes the
pressure, $W$ is a Gaussian random field white
in time, subject to the restrictions imposed below on the space correlation and
$G_\nu$ is an operator acting on the solution.
The velocity field $u_\nu$
is subject to the incompressibility condition
\begin{equation}\label{incomp_condition}
\nabla\cdot u_\nu(t,x)=0,\ \ \ t\in [0,T],\ \ \ x\in D,
\end{equation}
and to the boundary condition for every $t\in [0,T]$
\begin{equation}
u_\nu(t,.)\cdot n=0\ \ {\rm and}\ \ \mbox{\rm curl } u_\nu(t,.)=0\ \ \ {\rm on}\ \partial D,\label{boundary_condition}
\end{equation}
$n$ being the unit outward normal to $\partial D$. The initial
condition is the function $\zeta$ defined by:
\begin{equation}\label{initial_condition}
u_\nu(0,x)=\zeta(x), \ \ \forall  x\in D.
\end{equation}

We are interested in the asymptotic properties of the distribution of the
process $u_\nu(t,.)$ as the viscosity goes to 0. More precisely, the aim of the present paper  is to prove a Large Deviation Principle
(LDP) for the stochastic 2D Navier Stokes equations  \eqref{Euler}
when the viscosity coefficient $\nu \to 0$ and the noise $W$  is multiplied by the
square root of the viscosity, in order to be in the Freidlin-Wentzell setting.
A similar idea has been pursued by S. B. Kuksin in \cite{kuksin_cpm2008}, where he studied the convergence
of the invariant measure of the equation \eqref{Euler} when it is driven by an additive degenerate noise.
Indeed, Kuskin establishes asymptotic properties of this invariant measure  when the viscosity is small.
In this paper, we study the exponential concentration of the distribution of the process
$u_\nu(t,.)$ for a fixed $t$, when the viscosity decays to zero;
we hope to be able to extend this study for stationary solutions.

Several recent papers have studied a LDP for the distribution of the
solution to a hydro-dynamical stochastic evolution equation. We
refer to \cite{SrSu} for the 2D Navier-Stokes equations, \cite{DM}
for the Boussinesq  model, \cite{ChuMi} for more general
hydro-dynamic models, \cite{RocZha}   for tamed 3D Navier Stokes equations.
All the above papers consider an equation with
a (fixed) positive viscosity coefficient and study the exponential
concentration to a deterministic model when the noise intensity is
multiplied by a coefficient $\sqrt{\epsilon}$ which converges to 0.
They deal with a multiplicative noise and use the weak
convergence approach of LDP, based on the Laplace principle,
developed by P. Dupuis and R. Ellis in \cite{DupEl97}.

Reference \cite{BesMil} dealt with  a simpler
equation driven by a multiplicative noise and a vanishing viscosity coefficient, that is a shell
model of turbulence. Under certain conditions on the initial
condition and the operator acting on the noise, this equation is
well posed in ${\mathcal C}([0,T]; V)$ where $V$ is a Hilbert space similar to $H^{1,2}$.
A LDP was proved for a weaker topology,
that of $L^{2}(0,T; \mathcal{H})$, where $\mathcal{H}$ is a subspace
of $V$ similar to $H^{\frac{1}{2},2}$, with the same scaling between the "viscosity"
and the square of the noise intensity. The technique used was  again the weak convergence approach.
To our knowledge, this was  the first paper that proved  a
LDP when the coefficient in front of the
noise term depends on the viscosity and converges to 0. Let us  point out  that the study of
the inviscid limit is an important step towards understanding turbulent fluid flows in general.
Let us also refer to the paper of M. Mariani \cite{MaMa}, where a "nonviscous"  scalar equation is
considered in the context of conservation laws. However the
techniques used in that paper are completely different from the ones
used here and in \cite{BesMil}.

In this paper, we will generalize
our result to the Navier Stokes  equations  \eqref{Euler} in a bounded domain
of $\mathbb{R}^{2}$; this is technically more involved.
Here, the family $(G_\nu, \nu>0)$ of operators is of the form $G_\nu = \sqrt{\nu} \sigma_\nu$,
where  the family  $\sigma_\nu$ converges to $ \sigma_0$ in an appropriate topology as $\nu \to 0$. Similarly,
we can deal with a more general family $(\sigma_\nu , \nu >0)$ of gradient type converging to some more
regular operator $\sigma_0$ which is no longer of gradient type.
Gradient type noise is an active topic of research for turbulent flows; see e.g.
\cite{Rosovski_Mik} and the references therein. However, in order to focus on the main ideas of the inviscid
limit and avoid heavy tehnical computations, we choose to work with simpler $\sigma_\nu$.
Note that the rate function in this framework is described by the solution to
a deterministic  "controlled" Euler equation
\begin{equation} \label{contEuler}
\frac{\partial u(t)}{\partial t} + (u(t)\cdot\nabla)u(t)=-\nabla p+\sigma_0(t,u(t)) h(t),
\end{equation}
with the same incompressibility and boundary conditions, where $h$ denotes an
element of the Reproducing Kernel Hilbert Space (RKHS) of the noise.
This equation is a deterministic counterpart
of  the stochastic Euler equation studied
by \cite{BF} in the case
of additive noise, \cite{BP2001}  and \cite{Bes}  when the noise is multiplicative.
There is an extensive literature for the deterministic Euler
equation in dimension 2. We refer to
\cite{Bardos}, \cite{KaPo}, \cite{Yu} and the references therein and
\cite{Bardos-Titi} for a survey paper.

The technique we use is again the weak convergence approach and  will require
to prove well posedness and apriori bounds of the solution to \eqref{contEuler}
in the space
${\mathcal C}([0,T]; L^2)\cap L^\infty(0,T;H^{1,q})$ for
all $q>2$  and  for a more regular initial condition.
Thus, we are able to prove
the LDP  in a "non-optimal" space for the Navier Stokes equations with positive viscosity,
namely $L^2(0,T;{\mathcal H})$, where ${\mathcal H}$ is  a Hilbert interpolation space between
$H$ and $V$ similar to that in \cite{BesMil}.
This is due to the fact that the Euler equation has no regularizing effect on the solutions
and stronger conditions are required in order to have uniqueness of the solution;
this forces us to work with non Hilbert Sobolev spaces $H^{1,q}$ for $q\in (2,\infty)$ and to
require that the diffusion coefficient $\sigma$ is both trace class and Radonifying.
Indeed, some apriori estimates have to be obtained in general Sobolev spaces
uniformly in the "small" viscosity $\nu >0$  for the stochastic  Navier Stokes equations \eqref{Euler}
when the noise $W$ is multiplied by $\sqrt{\nu}$ and shifted by a random element
of its RKHS.

Let us finally point out that, even if the
problem solved here is similar to that in \cite{BesMil}, the final step is quite
different. Indeed, unlike all the references on LDP for hydrodynamical models,
the weak convergence is proven using a tightness argument and not by means of the convergence
in $L^2$ of a properly localized sequence. Unlike in \cite{BesMil}, no   time increment has to be studied
and no H\"older regularity of the map $\s(.,u)$ has to be imposed.  Let us also point
out that we replace the classical homogenous Dirichlet boundary conditions by the free boundary
one. Working with the classical homogeneous Dirichlet boundary condition  would lead
to some boundary layers problems that are beyond the scope of this paper.
For more details and explanations about the free boundary condition
\eqref{boundary_condition} we refer to \cite{Ziane}.  Let us also mention
that all our results can  be proved for the stochastic Navier-Stokes equations with periodic conditions.

The paper is organized as follows: In section 2 we describe the model and establish apriori
estimates in the Hilbert spaces $L^2$ and $H^{1,2}$ similar to known ones,
except for two things: the boundary conditions are slightly different, and we have
to prove estimates  uniform in a "small"
viscosity $\nu$. Section 3 deals with the inviscid problem in $C([0,T];L^2)
\cap L^\infty(0,T;H^{1,q})$. Section 4 proves apriori bouunds of the NS equations in $H^{1,q}$
and section 5 establishes the large deviations results. Finally, some technical results on Radonifying
and Nemytski's operators  are gathered in  the Appendix.

\section{Description of the model} \label{s2}

For every $\nu>0$, we
consider the equations of Navier-Stokes type
\begin{equation}\label{Navier-Stokes}
          \left\{\begin{array}{lr}
                \frac{\partial u}{\partial t} + (u\cdot\nabla) u + \nabla p =
                 \nu\Delta u + G_\nu (t,u)\frac{\partial W}{\partial t},
                 &{\rm in}\  [0,T]\times D, \\
                 \nabla\cdot u=0, &{\rm in}\  [0,T]\times D, \\
        \mbox{\rm curl } u=0\; \mbox{\rm and }  u\cdot n = 0 &{\rm on}\  [0,T]\times \partial D, \\
                 u|_{t=0}=\zeta,&{\rm in}\ D,
           \end{array}
     \right.
\end{equation}
where $\mbox{\rm curl} \, u
= D_{1}u_{2}-D_{2}u_{1}$.
\subsection{Notations and hypothesis}
Let $\mathcal{V}$ be the space of infinitely differentiable vector
fields $u$ on $D$ with compact support strictly contained in $D$,
satisfying $\nabla\cdot u=0$  in $D$ and $u.n=0$
on $\partial D$.  Let us denote by $H$ the closure of $\mathcal{V}$
in $L^{2}(D; \mathbb{R}^2)$,   that is
$$H = \left\{u\in\left[L^{2}(D)\right]^{2};\ \nabla\cdot u=0\ {\rm in}\
D,\ u\cdot n=0\ {\rm on}\ \partial D\right\}.$$ The space $H$ is a
separable Hilbert space with the inner product inherited from
$\left[L^{2}(D)\right]^{2}$, denoted in the sequel by $(.,.)$ and
$|.|_{H}$ denotes the corresponding norm.
For every integer $k\geq 0$ and any $q\in [1,\infty)$, let $W^{k,q}$ denote the completion
of the set of ${\mathcal C}^\infty_0(\bar{D},\RR)$ or of ${\mathcal C}^\infty_0(\bar{D},\RR^2)$
with respect to the norm
\[ \|u\|_{W^{k,q}}=\Big( \sum_{|\alpha|\leq k} \int_D |\partial^\alpha u(x)|^q\, dx \Big)^{\frac{1}{q}}.\]
To ease notations, let $\|.\|_q:=\|.\|_{W^{0,q}}$.
For $k<0$ and $q^*=q/(q-1)$, let $W^{-k,q^*}=(W^{k,q})^*$.
Here, for a multi-index $\alpha=(\alpha_{1},\alpha_{2})$  we set
$\displaystyle\partial^\alpha u(x)=\frac{\partial^{|\alpha|} u(x)}{\partial x_{1}^{\alpha_{1}}\partial x_{2}^{\alpha_{2}}}.$
For a non-negative real number  $s=k+r$, where $k$ is an integer and $0<r<1$, and for
any $q\in [1,\infty)$, let $W^{s,q}$ denote the completion
of the set of ${\mathcal C}^\infty_0(\bar{D},\RR)$ or of ${\mathcal C}^\infty_0(\bar{D},\RR^2)$
with respect to the norm defined by:
\[ \|u\|^{q}_{W^{s,q}}=\|u\|^{q}_{W^{k,q}}+
\sum_{|\alpha|=k}\int_{D}\int_{D}\frac{|\partial^{\alpha}u(x)-\partial^{\alpha}u(y)|^{q}}
{|x-y|^{2+2r}}dxdy.\]
Given $0<\alpha<1$,  let $W^{\alpha,p}(0,T;H)$ be the Sobolev
space of all $u\in L^{p}(0,T;H)$ such that
$$\displaystyle\int_{0}^{T}\int_{0}^{T}\frac{|u(t)-u(s)|^{p}}{|t-s|^{1+\alpha
p}}dtds<\infty.$$

Let us set $H^{k,q}=W^{k,q}\cap H$ for any $k\in [0,+\infty)$ and $q\in [2,\infty)$; the set
$H^{k,q}$ is endowed with the norm inherited from that of $W^{k,q}$ and denoted by
$\|.\|_{H^{k,p}}$. Let $V=H^{1,2}$, that is  the subspace of $H$ defined as follows:
$$V = \left\{u\in W^{1,2}(D;\RR^2): \ \nabla\cdot u=0\ {\rm in}\
D,\ u\cdot n=0\ {\rm on}\ \partial D\right\}.$$
The space $V$ is a separable Hilbert space with the inner product
$((.,))$ inherited from that of $W^{1,2}(D;\RR^2)$ and $\|.\|:= \|.\|_{V}$ denotes
the corresponding norm, defined  for $u,v\in V$ by:
\[ \|u\|^2 = ((u,u)) \, , \mbox{ \rm and } ((u,v)) = \int_D \big[ u(x) . v(x)  + \nabla u(x) . \nabla v(x) \big] dx .\]
Identifying $H$ with
its dual space $H'$, and $H'$ with the corresponding natural
subspace of the dual space $V'$, we have the  Gelfand triple
$V\subset H\subset V'$ with continuous dense injections. We denote
the dual pairing between $u\in V$ and $v\in V'$ by $\langle  u,v\rangle$.
When $v\in H$, we have $(u,v)=\langle u,v\rangle$.
Let $b(\cdot,\cdot,\cdot): V\times V\times V\longrightarrow
\mathbb{R}$ be the continuous trilinear form defined as
$$b(u,v,z)=\int_{D}(u(x)\cdot\nabla v(x))\cdot z(x)\, dx .$$
It is well known that there exists a continuous bilinear operator
$B(\cdot,\cdot): V\times V\longrightarrow V'$ such that
$\langle  B(u,v),z\rangle =b(u,v,z),\ {\rm for}\ {\rm all}\ z\in V.$
By the incompressibility condition, for $u,v, z\in V$
we have (see e.g. \cite{Lions70} or \cite{Bardos})
\begin{equation} \label{incompress}
\langle B(u,v),z\rangle=- \langle B(u,z),v\rangle \quad   \mbox{\rm and}\quad    \langle B(u,v),v\rangle =0.
\end{equation}
Furthermore, there exits a constant $C$ such that for any $u\in V$,

\begin{equation}\label{bilinear_estimate1}
\|B(u,u)\|_{V'} \leq C |u|_{H}\, \|u\|.
\end{equation}

Let $a(\cdot,\cdot): V\times V\longrightarrow {\mathbb R}$ be the
bilinear continuous form defined in \cite{Bardos} as
$$ a(u,v)= \int_{D}\nabla u\cdot\nabla v
-\int_{\partial D}k(r)u(r)\cdot v(r)dr,$$ where $k(r)$ is the
curvature of the boundary $\partial D$ at the point $r$, and we have
the following estimates (see \cite{Lions} for  details):
\begin{equation}\label{boundary_estimate}
\int_{\partial D}k(r)u(r)\cdot v(r)dr \leq C\| u\|\| v\|,
\end{equation}
and for any $\epsilon>0$ there exists  a positive constant $C(\epsilon)$ such that:
\begin{equation}\label{lions}
\int_{\partial D}k(r)|u(r)|^2 dr \leq \epsilon\| u\|^2
+C(\epsilon)|u|_{H}^2.
\end{equation}  
Moreover, we set
$D(A)=\left\{u\in  H^{2,2} \; :\; \mbox{\rm curl }
u=0\; \mbox{\rm on } \partial D \right\}$, and define the linear
operator $A:D(A)\longrightarrow H$ as
$$Au = -\Delta u, \; \mbox{\rm i.e.,}\;   a(u,v):=(A u,v).$$
On the other hand, for all $u\in D(A)$ we have
\begin{equation}\label{B(u,u)Au}
(B(u,u),Au)=0.
\end{equation}
For $\beta >0$ we will denote the $\beta$-power of the operator $A$  by $A^\beta$
and its domain by  $D(A^\beta)$. Here $D(A^{-\beta})$ denotes the
dual of $D(A^{\beta})$.   Note that for $k<3/4$, we have $H^{k,2}=D(A^{k/2})$; the proof can be found in \cite{BP2001}
Theorem 3.1.
Set $\mathcal{H}=H^{1/2,2}$ and  note that
$\mathcal{H}=D(A^{1/4})$ and $V=D(A^{1/2})$.
The continuous embedding  $V\subset \mathcal{H}\subset H$ holds.
Moreover,  $\mathcal{H}$ is an interpolation space, that is there exists a constant $a_{0}>0$ such that
\begin{equation}\label{interpolation_H}
\|u\|^{2}_{\mathcal{H}}\leq a_{0}  |u|_{H}  \|u\|,\ {\rm for}\ {\rm all}\
u\in V.
\end{equation}
Since   ${\mathcal H} \subset L^{4}(D)$
and $\langle B(u,v),w\rangle  = -\langle B(u,w), v\rangle $, we deduce
\begin{equation}\label{B(u,u)L4}
|\langle B(u,v),w\rangle |\leq C \|u\|_{\mathcal{H}}\|v\|_{\mathcal{H}}\|w\|,
\end{equation}
and $B$ can be extended as a bilinear operator from
$\mathcal{H}\times\mathcal{H}\longrightarrow V'$.

\noindent In place of equations (\ref{Navier-Stokes}) we will
consider the abstract stochastic evolution equation:
\begin{equation}\label{NS_abstract}
          du(t)+\nu Au(t)dt+B(u(t),u(t))dt=
          \sigma(t, u(t))dW(t)
\end{equation}
on the time interval $[0,T]$ with the initial condition $ u(0)=\zeta$
and $B$ satisfies
conditions \eqref{incompress}, \eqref{bilinear_estimate1}, \eqref{B(u,u)Au} and \eqref{B(u,u)L4}.

\subsection{Stochastic driving force}
Let $Q$ be a linear positive  operator in the Hilbert space $H$
which is  trace class, and hence   compact. Let $H_0 = Q^{\frac12}
H$; then $H_0$ is a Hilbert space with the scalar product
$$
(\phi, \psi)_0 = (Q^{-\frac12}\phi, Q^{-\frac12}\psi),\; \forall
\phi, \psi \in H_0,
$$
together with the induced norm $|\cdot|_0=\sqrt{(\cdot, \cdot)_0}$.
The embedding $i: H_0 \to  H$ is Hilbert-Schmidt and hence compact,
and moreover, $i \; i^* =Q$. Let $L_Q\equiv L_Q(H_0,H) $ be the
space of linear operators $S:H_0\mapsto H$ such that $SQ^{\frac12}$
is a Hilbert-Schmidt operator  from $H$ to $H$. The norm in the
space $L_Q$ is
  defined by  $|S|_{L_Q}^2 =tr (SQS^*)$,  where $S^*$ is the adjoint operator of
$S$. The $L_Q$-norm can also be  written in the form

\begin{equation}\label{LQ-norm}
|S|_{L_Q}^2=tr ([SQ^{1/2}][SQ^{1/2}]^*)=\sum_{k\geq 1} |SQ^{1/2}\psi_k|_H^2=
\sum_{k\geq 1} |[SQ^{1/2}]^*\psi_k|_H^2
\end{equation}
for any orthonormal basis  $(\psi_k)$ in $H$.
\par
Let   $(W(t), t\geq 0)$ be a   Wiener process  defined   on a filtered
probability space $(\Om, \cF, (\cF_t), \PX)$, taking values in $H$
and with covariance operator $Q$. This means that $W$ is Gaussian,
has independent time increments and that for $s,t\geq 0$, $f,g\in
H$,
\[
\EE  (W(s),f)=0\quad\mbox{and}\quad \EE  \big[ (W(s),f) (W(t),g)\big] =
\big(s\wedge t)\, (Qf,g).
\]
Let  $(\beta_j)$ be  standard (scalar) mutually independent Wiener
processes, $( e_j)$ be an  orthonormal basis in $H$ consisting of
eigen-elements of $Q$, with $Qe_j=q_je_j$. Then $W$ has  the
following  representation
\begin{equation}\label{W-n}
W(t)=\lim_{n\to\infty} W_n(t)\;\mbox{ in }\; L^2(\Om; H)\; \mbox{
with } W_n(t)=\sum_{1\leq j\leq n} q^{1/2}_j \beta_j(t) e_j,
\end{equation}
and $Trace(Q)=\sum_{j\geq 1} q_j$. For details concerning this
Wiener process see e.g.  \cite{PZ92}.\\

\noindent Let $k\geq 0$,   $q\in [2,\infty)$ and let  $R(H_{0},W^{k,q})$ denote the space of all
$\gamma$-radonifying mappings from $H_{0}$ into $W^{k,q}$, which are analogues of Hilbert-Schmidt
operators when the Hilbert Sobolev spaces $W^{k,2}$ are replaced by the more general Banach spaces
$W^{k,q}$. The
definitions and some basic properties of stochastic calculus in the
framework of special Banach spaces, including the case of non-Hilbert Sobolev spaces, can be found in
\cite{BP2001}; see also \cite{ZB}, \cite{Det}, \cite{N78} and \cite{O}.
For the sake of self-completeness, they are described
in sub-section \ref{Radon} of the Appendix. The radonifying norm $\| S\|_{R(H_{0},W^{k,q})}$ of an element
$S$ of $R(H_{0},W^{k,q})$ is defined in \eqref{normRadon} ; it is the extension of the $L_Q$ norm of $ S\in  L_Q$ which is the particular case $k=0$ and $q=2$.

\subsection{Assumptions}

Given a viscosity coefficient $\nu >0$, consider the following
stochastic Navier-Stokes equations
\begin{equation} \label{SNS}
du^\nu(t)+ \big[ \nu A u^\nu(t) + B(u^\nu(t), u^\nu(t))\big]\, dt = \sqrt{\nu}\,
\sigma_\nu(t, u^\nu(t))\, dW(t),
\end{equation}
where the noise intensity $\s_{\nu}: [0,T]\times V \to L_Q(H_0, H)$
of the stochastic perturbation is properly normalized by the square
root of the viscosity coefficient $\nu$. We assume that
  $\sigma_\nu$ satisfies the following growth and Lipschitz
conditions:
\par
\noindent \textbf{Condition (C1):} {\it For every $\nu>0$,  $\s_{\nu}\in {\mathcal
C}\big([0,T]\times V; L_Q(H_0, H)\big)$, there exist constants   $K_i, L_1 \geq 0$
such that for every $t\in [0,T]$, $\nu>0$  and $u,v\in V$:\\
{\bf (i)}   $|\s_\nu(t,u)|^2_{L_Q} \leq K_0+ K_1 |u|_{H}^2$, \\
{\bf (ii)}   $|\s_\nu(t,u)-\s_\nu(t,v)|^2_{L_Q} \leq   L_1 |u-v|_{H}^2 $.}
\smallskip

For technical reasons, in order to prove a large deviation principle
for the distribution of  the solution to \eqref{SNS} as the
viscosity coefficient $\nu$ converges to 0,
we will need some precise estimates
on the solution of the equation  deduced from \eqref{SNS} by
shifting
the Brownian $W$ by some
random element  of its RKHS.  This cannot be deduced from similar
ones  on $u$ by means of a Girsanov transformation; indeed,  the
Girsanov density is not uniformly bounded in  $L^2(P)$ when the intensity of the
noise tends to zero (see e.g. \cite{DM} or \cite{ChuMi}).
\par
To describe a set of admissible random shifts,   we introduce the
class
$\mathcal{A}$ as the  set of $H_0-$valued
$(\cF_t)-$predictable stochastic processes $h$ such that $\int_0^T
|h(s)|^2_0 ds < \infty, \; $ a.s. For fixed $M>0$, let
\[S_M=\Big\{h \in L^2(0, T; H_0): \int_0^T |h(s)|^2_0 ds \leq M\Big\}.\]
The set $S_M$, endowed with the following weak topology, is a
  Polish  (complete separable metric)  space
(see e.g. \cite{BD07}): $ d_1(h, k)=\sum_{k\geq 1} \frac1{2^k}
\big|\int_0^T \big(h(s)-k(s), \tilde{e}_k(s)\big)_0 ds \big|,$ where
$ ( \tilde{e}_k(s) , k\geq 1)$ is an  orthonormal basis for
$L^2(0, T ;  H_0)$. For $M>0$ set
\begin{equation} \label{AM}
\mathcal{A}_M=\{h\in \mathcal{A}: h(\om) \in
S_M, \; a.s.\}.
\end{equation}
In order to define the stochastic controlled equation, we introduce for
$\nu\geq 0$ a family of   intensity coefficients $\tilde{\s}_\nu$  which act
on a  random element $h\in {\mathcal A}_M$ for some $M>0$. The case
$\nu=0$ will be that of an inviscid limit "deterministic" equation
with no stochastic integral,  and which can be dealt with for fixed
$\omega$. We assume that for any $\nu\geq 0$ the coefficient
$\tilde{\sigma}_\nu$ satisfies the following condition,
similar to {\bf (C1)} and weaker since
the $L_Q$ norm is replaced by the smaller one of $L(H, H_0)$.
\medskip
\par
\noindent \textbf{Condition (C1Bis):}
{\it For any $\nu \geq 0 $,   ${\tilde \s}_\nu \in
{\mathcal C}\big([0,T]\times V; L(H_0, H)\big)$ and there exist non negative
constants  $\tilde{K}_i$ and $\tilde{L}_1$ such that
for every $t\in [0,T]$, $\nu \geq 0$ and $u,v\in V$:
\begin{align}
\label{growth-tilde}  |\tilde{\s}_\nu(t,u)|_{L(H_0,H)} \leq \tilde{K}_0 + \tilde{K}_1 |u|_H ,
\\
|\tilde{\s}_\nu(t,u) -\tilde{\s}_\nu(t,v)  |_{L(H_0,H)}
\leq \tilde{L}_1 |u-v|_H  .
\label{tilde-s-lip}
\end{align}
}
Examples of coefficients  $\sigma_\nu$ and $\tilde{\sigma}_\nu$ which satisfy conditions {\bf (C1)} and
{\bf (C1Bis)}, of Nemytski form,
are provided in subsection \ref{Nemitski_op} of the Appendix.

\par
Let $\nu>0$,  $M >0$,  $h\in {\mathcal A}_M$,  $\zeta$ be  an $H$-valued random
variable independent of $W$. Under Conditions   ({\bf
C1}) and ({\bf C1Bis}), we consider the nonlinear SPDE
\begin{align} \label{uhnu}
& d u_h^\nu(t)  + \big[ \nu\,  A u_h^\nu(t) +
B\big(u_h^\nu(t) , (u_h^\nu(t)\big) \big]\, dt = \sqrt{\nu}\, \sigma_\nu
(t,u_h^\nu(t))\, dW(t) + \tilde{\s}_\nu(t,u_h^\nu(t)) h(t)\,
dt,\nonumber \\
& u_h^\nu(0)=\zeta.
\end{align}

Well posedness of the above equation as well as apriori
bounds of the solution to this equation in ${\mathcal
C}([0,T];H)\cap L^2(0,T;V)$ are known for fixed  $\nu>0$ when $u=0$
on $\partial D$ (see e.g. \cite{SrSu} and \cite{ChuMi}. We will prove
them uniformly in $\nu\in (0,\nu_0]$ for some small $\nu_0$ under
different boundary conditions.

\par
Let us  introduce the following conditions that we will use later in the
paper.  The following conditions \textbf{(C2)} and \textbf{(C2Bis)}
will allow to improve apriori estimates on the $p$-th moment of the
solution to the  stochastic controlled equation \eqref{uhnu} in $V$,
uniformly in time and on a "small"  viscosity coefficient $\nu$.
They will also yield the existence of a solution to the inviscid
deterministic equation, that is of \eqref{uhnu} when $\nu =0$.
\par
\noindent \textbf{Condition (C2):} {\it  For every $\nu >0$,
$\s_\nu  \in {\mathcal C}\big([0,T] \times D(A); L_Q(H_0, V)\big)$
and there exist non negative constants $K_i, L_1 $
such that for every $t\in [0,T]$, $\nu>0$ and $u,v\in D(A)$:\\
{\bf (i)}  $| \mbox{\rm curl } \s_\nu(u)|^2_{L_Q} \leq K_0+ K_1 \|
u\|_V^2 $,}\\
{\bf (ii)} $|A^{1/2}\s_\nu(t,u)-A^{1/2}\s_\nu(t,v)|^2_{L_Q} \leq L_1
\|u-v\|_V^2 $.

\noindent \textbf{Condition (C2Bis):} {\it For every $\nu \geq 0$,
$\tilde{\s}_\nu \in {\mathcal C}\big([0,T]\times D(A) ; L(H_0,  V)\big)  $,  
there exist   non negative constants
$\tilde{K}_i$, $\tilde{L}_1$,
such that for every $t\in [0,T]$, $\nu\geq 0$ and $u,v\in D(A)$:\\
{\bf (iii)} $|\mbox{\rm curl }  \tilde{\s}_\nu(t,u)|_{L(H_{0},H)} \leq
\tilde{K}_0+ \tilde{K}_1\| u\|_V $,}\\
{\bf (iv)} $|A^{1/2}\tilde{\s}_\nu(t,u) -A^{1/2}\tilde{\s}_\nu(t,v)
|_{L(H_0,H)} \leq \tilde{L}_1 \|u-v\|_V .$
\par
Again, sub-section \ref{Nemitski_op} of the Appendix
provides examples of Nemytski operators  which satisfy all the conditions above.
\par
\subsection{Well Posedness and a priori estimates}
Let us mention  in this section that the results used to obtain the well posedness
of solutions are similar to known ones with different boundary conditions.
However the apriori estimates are more involved since we are seeking estimates uniform
in the parameter $\nu>0$ which  will be used later in Section 5 to let $\nu \to 0$.
Note that the results in this section would still be valid under  more general assumptions
than those stated in Conditions {\bf (C1)-(C2Bis)}, similar to that in \cite{DM} and \cite{ChuMi}.
The corresponding Nemitsky operators defining $\s_{\nu}$ and $\tilde{\s}_{\nu} $
could include some gradient of the solution multiplied by the square root
of the viscosity coefficient. However, to focus on the
main contribution of the present paper compared with previous related works, we prefer to keep
simpler and more transparent assumptions on the diffusion coefficient
$\s_\nu$ and an unrelated coefficient $\tilde{\s}_\nu$.

We at first recall  that an  $(\cF_t)$-predictable stochastic process $u^\nu_h(t,\om)$ is called a
{\em  weak
solution }  in $X\subset {\mathcal C}([0, T]; H) \cap L^2(0, T; V)$  
for the stochastic equation \eqref{uhnu}  on $[0, T]$
with initial condition $\zeta$ if
$u_h^\nu\in X$ 
a.s., and satisfies  a.s. the equality
\begin{align}\label{weak_form}
(u_{h}^{\nu}(t)  , v) & -(\zeta,v)
+\int_{0}^{t}\big[ \nu ( u_{h}^{\nu}(s),Av ) +\langle B(u_{h}^{\nu}(s),v),u_{h}^{\nu}(s)\rangle\big]ds\nonumber \\
=&\sqrt{\nu}\int_{0}^{t}(\sigma_{\nu}(s,u_{h}^{\nu}(s))dW(s),v)
+\int_{0}^{t}(\tilde{\sigma}_{\nu}(s,u_{h}^{\nu}(s))h(s),v) ds.
\end{align}
for all $v \in Dom(A)$ and all $t \in [0,T]$.
Note that this solution is a strong one in the probabilistic meaning, that is written
in terms of stochastic integrals with respect to  the given Brownian motion $W$.

\begin{Prop}\label{lemma1} Let  $T>0$, $(\sigma_\nu , \nu >0)$ and
$(\tilde{\sigma}_\nu, \nu >0)$ satisfy conditions ({\bf C1}) and ({\bf C1Bis})  respectively
and let
the initial condition $\zeta$ be such that ${\EX  } |\zeta|_{H}^{2p}<\infty$
for some $p\geq 2$. Then for any $M>0$ and $\nu_0>0$, there exist positive constants
$C_{1}(p,M)$ and $\tilde{C}_{1}(M)$ (depending also on  $T$, $\nu_0$,  $K_i, \tilde{K}_i,
i=0,1 , 2$,  such that for any $\nu \in (0,\nu_0 ]$ and
any  $h\in \mathcal{A}_{M}$,  \eqref{uhnu}  has a unique weak  solution in
${\mathcal C}([0,T];H)\cap L^2(0,T;V)$ which satisfies
the following apriori estimates:
\begin{equation}\label{uhnu_estimate1}
\sup_{0< \nu \leq \nu_0}\sup_{h\in {\mathcal A}_M}  {\EX   }\Big(\sup_{0\leq s\leq T}|u_h^\nu(s)|_{H}^{2p}\Big)
\leq C_{1}(p,M)\big[ 1+{\EX   } |\zeta|_{H}^{2p}\big],
\end{equation}
and
\begin{equation}\label{uhnu_estimate2}
\sup_{0<\nu\leq \nu_0}   \sup_{h\in {\mathcal A}_M}  \nu \int_{0}^{T}{\EX   }\big(  \|u_h^\nu(s)\|^2
+ \|u_h^\nu(s)\|_{\mathcal H}^{4}\big)ds\leq \tilde{C}_{1}(M)\big[ 1+{\EX   } |\zeta|_{H}^4\big] .
\end{equation}
\end{Prop}
\begin{proof}
The proof, which is quite  classical, requires some Galerkin approximation  of $u^\nu_h$, say $u^{\nu,n}_h$, for which
apriori estimates are proved uniformly in $n$.  Note that in our situation, these apriori estimates have
to be obtained  uniformly in $\nu \in (0,\nu_0]$ and $h\in {\mathcal A}_M$. 
Using a subsequence of  $(u^{\nu,n}_h, n\geq 1)$ which converges in the weak or the weak-star
topologies of appropriate spaces,
one can then  prove that there exists a solution to
\eqref{uhnu} (see e.g \cite{ChuMi} or \cite{SrSu}). The proof of the uniqueness is standard and omitted.
To ease notation, we  replace the Galerkin approximation
by the limit process $u^\nu_h$ to obtain the required apriori estimates uniformly in $n\geq 1$ and in
$\nu \in (0,\nu_0]$ for some $\nu_0>0$  under slightly more general boundary
conditions; the proof can then be completed as in the appendix of \cite{ChuMi}.
If the well-posedeness is already known, we use the solution $u^\nu_h$ instead of the
Galerkin approximation.
Let $\nu>0$, $h\in {\mathcal A}_M$; for  every $N>0$, let
$ \tau_{N}=\inf\left\{t\geq 0 ,\ \ |u_h^\nu(t)|_{H}\geq N\right\}\wedge T$.

Applying  It\^{o}'s formula first to $|.|_{H}^2$ and the process $u_h^\nu(.\wedge \tau_N)$,
then to the map $x\mapsto x^p$ for
$p\geq 2$ and the process $|u^\nu_h(.\wedge \tau_N)|_{H}^2$,
we deduce:
\begin{equation}\label{ItoH}
|u_h^\nu(t\wedge\tau_{N})|_{H}^{2p}+\nu 2 p\int_{0}^{t\wedge\tau_{N}}|u_h^\nu(s)|_{H}^{2p-2}
\| u^\nu_h(s)\|^2 \, ds\leq |u_h^\nu(0)|_{H}^{2p} + J(t) + \sum_{i=1}^5 T_i(t),
\end{equation}
where
\begin{eqnarray*}
J(t)&=&2p \sqrt{\nu}
\int_{0}^{t\wedge\tau_{N}}
|u_h^\nu(s)|_{H}^{2p-2} \, (\sigma_\nu(s,u_h^\nu(s))dW(s),u_h^\nu(s)) ,\\
T_{1}(t)&=&2 p \nu \int_{0}^{t\wedge\tau_{N}}|u_h^\nu(s)|_{H}^{2p-2}\int_{\partial
D}k(r)|u_h^\nu(r)|_{H}^{2}drds,\\
T_{2}(t)&=& 2p \int_{0}^{t\wedge\tau_{N}}|u_h^\nu(s)|_{H}^{2p-2} \langle B(u^\nu_h(s),u^\nu_h(s)) , u^\nu_h(s)\rangle
ds,\\
T_{3}(t)&=&2p \int_{0}^{t\wedge\tau_{N}}|u_h^\nu(s)|_{H}^{2p-2}
\big( \tilde{\s}_\nu(s,u_h^\nu(s))h(s),u_h^\nu(s) \big)ds,\\
T_{4}(t)&=&{\nu}p\int_{0}^{t\wedge\tau_{N}}|u_h^\nu(s)|_{H}^{2p-2}|\s_\nu(s,u_h^\nu(s))|^2_{L_Q}ds,\\
T_{5}(t)&=&2 {\nu}p(p-1)\int_{0}^{t\wedge\tau_{N}}|\s_\nu^*(s,u_h^\nu(s)) u_h^\nu(s)|_0^2\,
|u_h^\nu(s)|_{H}^{2(p-2)}ds . \\
\end{eqnarray*}
The incompressibility condition \eqref{incompress} implies that $T_2(t)=0$ for any $t\in [0,T]$.
Using \eqref{lions}, we deduce that for any $\epsilon >0$ there exists a constant $C(\epsilon)$ such that
\begin{equation*}
T_{1}(t)\leq 2 \nu p\epsilon\int_{0}^{t\wedge\tau_{N}}|u_h^\nu(s)|_{H}^{2p-2}
\|u^\nu_h(s)\|^2 ds+ 2 \nu p\, C(\epsilon)\int_{0}^{t\wedge\tau_{N}}|u_h^\nu(s)|_{H}^{2p}ds.
\end{equation*}
Since $h\in {\mathcal A}_M$, the growth condition \eqref{growth-tilde},
the Cauchy-Schwarz and H\"older inequalities
imply:
\begin{align*}
T_{3}(t)  &
\leq
2p \int_{0}^{t\wedge\tau_{N}}\Big[ {\tilde{K}_0} + \Big({{\tilde K}_0 } + {{\tilde K}_1}\Big) |u^\nu_h(s)|_{H}^{2p}
\Big] |h(s)|_0 ds\\
&\;\leq 2p {\tilde{K}_0} \sqrt{M T} + 2p  \Big({{\tilde K}_0 } + {{\tilde K}_1}\Big) \int_0^{t\wedge \tau_N}
|u^\nu_h(s)|_{H}^{2p} |h(s)|_0 ds .
\end{align*}
Using the growth condition {\bf (C1)}, we deduce for $\nu \in (0,\nu_0]$:
\[  
T_4(t) + T_5(t) \leq \nu p (2p-1) K_0 T + \nu p (2p-1) (K_0+K_1)
\int_0^{t\wedge \tau_N}   |u^\nu_h(s)|_{H}^{2p} ds.
\]
Thus, the It\^o formula \eqref{ItoH} and the previous upper estimates of $T_i(t)$, $i=1, \cdots, 5$,
imply that for any $t\in [0,T]$,    $\epsilon\in (0,1)$,  
\begin{align}\label{ItoHbis}
& |u^\nu_h(t\wedge \tau_N )|_{H}^{2p}
+ 2 \nu p \big( 1 - \epsilon
\big) \int_0^{t\wedge \tau_N}
|u^\nu_h(s)|_{H}^{2p-2} \, \|u^\nu_h(s)\|^2 \, ds \nonumber \\
&\qquad  \leq \tilde{Z} + \int_0^t \tilde{\varphi}(s)  |u(s\wedge \tau_N)|_{H}^{2p} ds + J(t),
\end{align}
where
\begin{align*}
\tilde{Z}&= |\zeta|_{H}^{2p} + 2 p \tilde{K}_0 \sqrt{ M T} + p(2p-1) \nu K_0 T, \\
\tilde{\varphi}(s)&=p \Big[ 2\nu   C(\epsilon)  + (2p-1) \nu (K_0+ K_1)
+ 2 \Big( {{\tilde K}_0 }
  + {{\tilde K}_1}\Big) |h(s)|_0
\Big] .
\end{align*}
For $t\in [0,T]$, set
\[ X(t):=\sup_{0\leq s\leq t} |u^\nu_h(s\wedge \tau_N)|_{H}^{2p},\;
Y(t):=\int_0^{t\wedge \tau_N} |u^\nu_h(s)|_{H}^{2p-2}\|u^\nu_h(s)\|^2 ds ,\;
\tilde{I}(t):=\sup_{0\leq s\leq t} J(s).
\]
Let $\e =\frac{1}{2}$, $\nu \in (0, \nu_0]$, $\lambda \in (0,1)$ and $\tilde{\alpha}= (1-\lambda) \nu p$.
With these notations, the inequality \eqref{ItoHbis} yields
\begin{equation} \label{ItoH3}
\lambda X(t) + (1-\lambda) |u^\nu_h(t\wedge \tau_N)|_{H}^{2p} + \tilde{\alpha} Y(t)
\leq \tilde{Z} + \int_0^t \tilde{\varphi}(s) X(s)\, ds + \tilde{I}(t).
\end{equation}
Furthermore, using the Burkholder-Davis-Gundy inequality, condition {\bf (C1)}, then Cauchy-Schwarz's and
Young's inequalities, we deduce that for any $\beta >0$,
\begin{align*}
  {\EX   } \tilde{I}(t)& \leq
6 \sqrt{\nu} p {\EX   } \left( X(t)
\int_{0}^{t\wedge\tau_{N}}\!\!\!\!  \big[ K_{0}
+( K_0+K_1) |u_h^\nu(s)|_{H}^{2p}
\big] ds\right)^{1/2}\\
&\leq \tilde{\beta} {\EX   } X(t) + \tilde{\gamma} {\EX} \int_0^t X(s) ds
+ \bar{C},
\end{align*}
where
$\tilde{\gamma} = \frac{9\nu p^2}{ \tilde{\beta}}(K_0+K_1)$
and $ \bar{C}=
\frac{9\nu p^2}{\tilde{\beta}} K_0 T$.  
Let $\lambda = \frac{1}{2}$,  $ \varphi = 2 \tilde{\varphi}$,
$\alpha=2 \tilde{\alpha}$,  $\beta=2 \tilde{\beta}$, $\gamma=2 \tilde{\gamma}$ and $I(t)=\tilde{I}(t)$.
Then for $t\in [0,T]$, we have a.s.
\[  X(t) + {\alpha} Y(t) \leq  2 \tilde{Z}  
+ \int_0^t {\varphi}(s) X(s)\, ds + {I}(t), \quad   {\EX   } {I}(t)\leq
{\beta} {\EX   } X(t) + {\gamma} {\EX} \int_0^t X(s) ds + 2 \bar{C}.  \]
Furthermore, for $\nu \in (0,\nu_0]$  and $h \in {\mathcal A}_M$, one has a.s.  $\int_0^T \varphi(s) ds \leq \Phi(M, \nu_0)$,
where   $\Phi(M,\nu_0) = 4p (\tilde{K}_0 + \tilde{K}_1) \sqrt{MT} + 2p \nu_0 \big[ 2 C(\frac{1}{2} ) + (2p-1) (K_0+K_1)\big]$.
Let $\tilde{\beta} >0$ be such that $4\tilde{\beta} \exp(\Phi(M,\nu_0)) \leq 1$.  Then
since $X(.)$ is bounded by $N$,   Lemma A.1 in \cite{ChuMi}  (see also Lemma  3.9 in \cite{DM})
implies that for
$t\in [0,T]$, we have:
\[ {\EX   } \big[  X(t)+\alpha Y(t) \big] \leq C({\EX}|\zeta|_{H}^{2p}
,\nu_0, M,T),
\]
for some constant $C({\EX   }|\zeta|_{H}^{2p},
\nu_0, M,T)$ which does not depend on
$N$, $\nu\in (0,\nu_0]$, $h\in {\mathcal A}_M$
and on the step $n$ of the Galerkin approximation.
Since the right handside in the above equation does not depend on $N$, letting $N\to \infty$ we obtain that
$\tau_N\to T$ a.s. Hence there exists a constant $C_1:=C_1(\EX|\zeta|_H^{2p},\nu_0,M,T)$ such that
the Galerkin approximation $u^{n,\nu}_h$ of $u^\nu_h$ satisfies:
\[\sup_{n\geq 1} {\EX   } \Big( \sup_{0\leq t\leq T} |u^{n,\nu}_h(t)|_{H}^{2p} + \nu
\int_0^T \big[ \|u^{n,\nu}_h(t)\|_{\mathcal H}^4 +
\|u^{n,\nu}_h(s)\|^2 \big]  ds \Big) \leq C_1
\]
for any $n$,  $\nu \in (0,\nu_0]$ and  $h\in  {\mathcal A}_M$.
The proof is completed using a classical argument (see e.g. the Appendix of \cite{ChuMi} for details.)
\end{proof}

\begin{Prop}\label{lemma2} Let  the assumptions of  Proposition
\ref{lemma1} be  satisfied for $p=1$ or some $p\in [2,\infty)$. Moreover, assume that the initial
condition $\zeta$ is such that $E\|\zeta\|^{2p}<\infty$ and
that $(\sigma_\nu , \nu >0)$ and $(\tilde{\sigma}_\nu , \nu >0)$ satisfy  respectively conditions {\bf (C2)} and
{\bf (C2Bis)}.  Then
given  any $M>0$, there exists  a positive constant
$C_{2}(p,M)$  such that for
  $\nu\in (0, \nu_{0}]$ and $h\in \mathcal{A}_{M}$,
the solution to \eqref{uhnu} satisfies:
\begin{equation}\label{uhnu_estimate3}
{\EX   }\Big( \sup_{0\leq t\leq T} \|u_h^\nu(t)\|^{2p} + \nu \int_0^T
| A u^\nu_h(s)|_{H}^{2}   \;  ds
\Big) \leq C_{2}(p,M) \big( 1+ {\EX   } \|\zeta\|^{2p}\big) .
\end{equation}
\end{Prop}
\begin{proof}
Let $\xi_{h}^{\nu}=\mbox{\rm curl } u_{h}^{\nu}$, then  it is a classical result that $u_{h}^{\nu}$ is
solution of the following elliptic problem   (see e.g.  \cite{Bes} and the references therein),
\begin{equation}\label{elliptic}
          \left\{\begin{array}{ll}
          -\Delta u_{h}^{\nu} = \nabla^{\bot}\xi_{h}^{\nu} & {\rm in}\ D,\\
           u_{h}^{\nu}\cdot n =  \xi_{h}^{\nu} = 0  & {\rm on}\ \partial D,
           \end{array}
   \right.
\end{equation}
where $\nabla^{\bot}=(D_{2}, -D_{1})$.
Using the equation   $\eqref{elliptic}$,  we get that
\begin{align*}
-( \Delta u_{h}^{\nu},\Delta u_{h}^{\nu})  &= (\nabla^{\bot}\xi_{h}^{\nu},\Delta u_{h}^{\nu} )     =
-( \nabla^{\bot}\xi_{h}^{\nu},\nabla^{\bot}\xi_{h}^{\nu} ) .
\end{align*}
Hence
\begin{align*}
|\Delta u_{h}^{\nu}|_H^{2}=|\nabla^{\bot}\xi_{h}^{\nu}|^{2}_{H}
=\|D_{2}\xi_{h}^{\nu}\|^{2}_{L^{2}(D)}+\|D_{1}\xi_{h}^{\nu}\|^{2}_{L^{2}(D)}
=|\nabla\xi_{h}^{\nu}|_H^{2}.
\end{align*}
Using \eqref{normcurl}  we see that the proof of \eqref{uhnu_estimate3} reduces to check  that
there exists a constant
$C(M, T,{K}_i,\tilde{K}_i):=C_{3}$ such that
for any $\nu\in (0,\nu_0]$ and $h\in {\mathcal A}_M$,
\begin{equation}\label{xihnu_p}
{\EX   } \Big(\sup_{0\leq t\leq T}  |\xi_{h}^{\nu}(t)|_{H}^{2p}
+ \nu \int_0^T |\nabla \xi^\nu_h (s)|_{H}^2 ds
\Big)\leq C_3 (1+{\EX   } |\mbox{ \rm curl } \zeta|_{H}^{2p}).
\end{equation}
We at first prove this inequality for the Galerkin approximation of the solution ; a standard argument
extends it to $u^\nu_h$ and hence $\xi_h^\nu$.
Fix $N>0$ and set $\bar{\tau}_N = \inf\{ t\geq 0 \, :\, |\xi_h^\nu(t)|_{H}\geq N\}\wedge T$.
Applying the {\rm curl} to the evolution equation \eqref{uhnu} yields $\xi_h^\nu(0)=\mbox{\rm curl }\zeta$ and
\begin{align} \label{xihnu}
d\xi_{h}^{\nu}(t) +  \nu A \xi_{h}^{\nu}(t)dt +  &
\mbox{\rm curl } B(u_{h}^{\nu}(t),u_{h}^{\nu}(t)) dt=  \nonumber \\
& \sqrt{\nu}\,
\mbox{\rm curl } \sigma_\nu
(s,u_h^\nu(t))\, dW(t) +
\mbox{\rm curl } \tilde{\s}_\nu (s,u_h^\nu(t))
h(t)\,dt .
\end{align}
Recall that equation \eqref{curlBnul} with $q=2$,
implies  $(\,  \mbox{\rm curl } B(u^\nu_h,u^\nu_h) , \xi_h^\nu) =0$
for $u\in D(A)$.
Using It\^o's formula for the square
of the $H$ norm, and then  for the map $x\to |x|_{H}^p$ with  $p\in [2,\infty)$,
we obtain  for $t\in [0,T]$:
\begin{equation}\label{Ito_xinnu}
|\xi_{h}^{\nu}(s\wedge \bar{\tau}_N)|_{H}^{2p} + 2p \nu \int_0^{t \wedge \bar{\tau}_N} \!\!
|\nabla \xi_h^\nu(s)|_{H}^2 \, |\xi_h^\nu(s)|_{H}^{2p-2} ds = |\mbox{\rm curl }\zeta|_{H}^{2p} + \bar{J}(t) + \sum_{i=1}^3\bar{T}_i(t),
\end{equation}
where
\begin{eqnarray*}
\bar{J}(t)&=& 2p \sqrt{\nu} \int_0^{t\wedge \bar{\tau}_N} |\xi_h^\nu(s)|_{H}^{2p-2}\;
\big( \mbox{\rm curl } \sigma_\nu(s,u^\nu_h(s)) dW(s)\, ,\,
\xi_h^\nu(s)\big) , \\
\bar{T}_1(t)&=&   2p \int_0^{t\wedge \bar{\tau}_N} |\xi_h^\nu(s)|_{H}^{2p-2}\;
\big( \mbox{\rm curl } \tilde{\sigma}_\nu(s,u^\nu_h(s)) h(s)\, ,\,
\xi_h^\nu(s)\big) \, ds , \\
\bar{T}_2(t)&=& \nu p \int_0^{t\wedge \bar{\tau}_N} |\xi_h^\nu(s)|_{H}^{2p-2} \;  |
\mbox{\rm curl } \sigma_\nu(s,u^\nu_h(s))|^2_{L_Q} \, ds,  \\
\bar{T}_3(t)&=& 2\nu p (p-1)  \int_0^{t\wedge \bar{\tau}_N} |\xi_h^\nu(s)|_{H}^{2p-4} \;
\big|  \big( \mbox{\rm curl } \sigma_\nu(s,u^\nu_h(s)) \big) ^* \xi_h^\nu(s) \big|^2_{H_0}  \, ds.
\end{eqnarray*}
Using the Cauchy-Schwarz inequality,   ({\bf C2Bis}) and \eqref{normcurl} with $q=2$, we get
that
\begin{align*}
\bar{T}_1(t)&\leq  2p  \int_{0}^{t\wedge{\bar \tau}_{N}}  |\xi_{h}^{\nu}(s)|_{H}^{2p-1}\, |\mbox{\rm curl }
\tilde{\s}_\nu(s,u)|_{L(H_{0},H)}\, |h(s)|_{0}\, ds\\
&\leq 2p  \int_{0}^{t\wedge {\bar \tau}_{N}} \Big[ {\tilde{K}_0} + \Big(  {\tilde{K}_0} + 2\, C {\tilde{K}_1} \Big)
|\xi_h^\nu(s)|_{H}^{2p}
+ {\tilde{K}_1} |u^\nu_h(s)|_{H} |\xi_h^\nu|_{H}^{2p-1}
\Big]
\, |h(s)|_{0}\, ds.
\end{align*}
Using Cauchy-Schwarz's, H\"older's and Young's inequalities, we deduce:
\begin{align*}
\bar{T}_1(t)&\leq
2p \tilde{K}_0 \sqrt{ M T} +  \tilde{K}_1^{2p}\sqrt{ MT} \, \sup_{0\leq s\leq T} |u^\nu_h(s)|_{H}^{2p} +
\int_{0}^{t\wedge{\bar \tau}_{N}} \psi_1(s) |\xi_h^\nu(s)|_{H}^{2p} ds ,  
\end{align*}
where
$\psi_1(s):= 2p\Big( {\tilde{K}_0} +2\,  C { \tilde{K}_1}\Big) + (2p-1)|h(s)|_0 $.

\noindent Furthermore, $\bar{T}_3(t)$ can be upper estimated in terms of  $\bar{T}_2(t)$ as follows:
\[ \bar{T}_3(t)\leq 2\nu p(p-1) \int_{0}^{t\wedge{\bar \tau}_{N}}  |\mbox{\rm curl } \sigma_\nu(u^\nu_h(s))|^2_{L_Q} |\xi_h^\nu(s)|_{H}^{2p-2}\, ds =
2(p-1) \bar{T}_2(t).\]
Finally, condition {\bf (C2)}, \eqref{normcurl}  with $q=2$, H\"older's and Young's inequalities, we obtain for $\nu\in (0,1]$:
\begin{align*}
\bar{T}_2(t) + \bar{T}_3(t) & \leq
\nu p (2p-1) \! \int_{0}^{t\wedge{\bar \tau}_{N}}\!\!\!
|\xi_h^\nu(s)|_{H}^{2p-2} \Big[ K_0 + K_1\big( |u^\nu_h(s)|_{H}^2 + 4  C^2
|\xi_h^\nu(s)|_{H}^2\big) \Big]\, ds\\
& \leq   \nu  (2p-1) T \Big[p K_0  + {K_1}  \sup_{0\leq s\leq T} |u_h^\nu(s)|_{H}^{2p} \Big]
+  \int_{0}^{t\wedge{\bar \tau}_{N}} \nu \, \psi_2 \, |\xi_h^\nu(s)|_{H}^{2p} \, ds,  
\end{align*}
where
$\psi_2 = (2p-1) \big[ p(K_0+ 4 K_1 C^2) + (p-1) K_1\big]$.
Let
\[ \bar{X}(t)=\sup_{0\leq s\leq t} |\xi_h^\nu(s\wedge \bar{\tau}_N)|_{H}^ {2p}\, ,
\quad  \bar{Y}(t)=\int_0^{t\wedge \bar{\tau}_N }\!\!
|\xi_h^\nu(s)|_{H}^{2p-2}\, |\nabla \xi_h^\nu(s)|_{H}^2\, ds. \]
Then for $\bar{\alpha}=2\nu p (1-\bar{\lambda})$,   $\bar{I}(t) = \sup_{0\leq s\leq t}|\bar{J}(s)|$, $h\in {\mathcal A}_M$
and
\begin{align*}
\bar{Z} &:=  |\mbox{\rm curl } \zeta|_{H}^{2p} + 2p\tilde{K}_0 \sqrt{MT}
+  \nu p (2p-1) T K_0  \\
&\quad + \big( \tilde{K}_1^{2p} \sqrt{ MT}  + \nu  (2p-1) T K_1\big) \sup_{0\leq s\leq T}
|u^\nu_h(s)|_{H}^{2p},
\end{align*}
equation \eqref{Ito_xinnu} and the upper bounds of $\bar{T}_i(t)$
imply that for
$t\in [0,T]$ and $\nu \in (0, \nu_0]$,  
\begin{equation}\label{majocurl2}
\bar{\lambda}  \bar{X}(t) + \bar{\alpha}  \bar{Y}(t) + (1-\bar{\lambda}) |\xi_h^\nu(s\wedge \bar{\tau}_N)|_{H}^{2p}
\leq \bar{Z} + \int_0^t \big[ \psi_1(s)+\nu \psi_2\big]  \bar{X}(s) ds + \bar{I}(t),
\end{equation}

\noindent The Davies inequality, condition {\bf (C2)}, \eqref{normcurl} for $q=2$,
Cauchy-Schwarz's, H\"older's  and Young's inequalities  imply  that for any $\bar{\beta} >0$,
\begin{align}
{\EX   }  \bar{I}(t)&\leq  6 p \sqrt{ \nu}  {\EX } \Big(  \int_{0}^{t\wedge{\bar \tau}_{N}}
|\xi_h^\nu(s)|_{H}^{4p-2} \; |\mbox{\rm curl }\sigma_\nu(u^\nu_h(s)|^2_{L_Q}\,
ds \Big)^{\frac{1}{2}} \nonumber \\
&\leq \bar{\beta}  {\EX   } \bar{X}(t) + \frac{9p^2 \nu}{\bar{\beta}} K_0 T
+    \frac{9p^2 \nu}{\bar{\beta}} (K_0 + 4 C^2 K_1) {\EX   } \int_0^{t\wedge \bar{\tau}_N}
|\xi_h^\nu(s)|_{H}^{2p} ds
\nonumber \\
&\qquad \qquad +  \frac{9p^2 \nu}{\bar{\beta}} K_1
\Big(  {\EX   } \int_0^{t\wedge \bar{\tau_N}} |\xi_h^\nu(s)|_{H}^{2p} ds \Big)^{\frac{p-1}{p}}
\Big(T  {\EX   } \sup_{0\leq s\leq T} |u^\nu_h(s)|_{H}^{2p}\Big)^{\frac{1}{p}} \nonumber \\
&\leq  \bar{\beta}  {\EX   } \bar{X} (t) + \bar{\gamma} {\EX }  \int_0^{t\wedge{\bar \tau}_N}
\bar{X}(s) ds  + \tilde{Z},
\end{align}
where
$ \bar{\gamma}= \frac{9p^2\nu}{\bar{\beta}} \big[ p(K_0+4 C^2 K_1) + (p-1) K_1\big]$
and $\tilde{Z}:= \frac{9p \nu T}{\bar{\beta}}
\big[p K_0  + K_1  {\EX   } \big(\sup_{0\leq s\leq T} |u^\nu_h(s)|_{H}^{2p} \big)\big]$.
Set $\bar{\lambda} = \frac{1}{2}$,
$\varphi(s)= 2 \big(\psi_1(s)+\nu \psi_2\big)$, $\alpha= 2\bar{\alpha}$,
$\beta= 2 \bar{\beta}$,  $\gamma= 2\bar{\gamma}$,
and $I(t)=2 \bar{I}(t)$.
Then for $\nu \in (0, \nu_0]$, $h\in {\mathcal A}_M$ and $t\in [0,T]$,
we have:
\begin{align*}  \bar{X}(t) + \alpha \bar{Y}(t) &\leq
\int_0^t \varphi (s) \bar{X}(s) ds + I(t) + 2 {\bar{Z}(t)}, \\
{\EX   } I(t) &\leq {\beta}  {\EX} \bar{X} (t) + {\gamma} {\EX }  \int_0^t \bar{X}(s) ds
+ {\delta} {\EX} \bar{Y}(t)
+ 2 { \tilde{Z}} .
\end{align*}

Furthermore,
there exists a constant $\Psi(M,\nu_0)>0$ such that
almost surely,
$\int_0^T \varphi(s) ds \leq  \Psi(M,\nu_0)$ for $\nu \in (0,\nu_0]$ and $h\in {\mathcal A}_M$.
Let $\bar{\beta}>0$ be such that $4 \bar{\beta} \exp(\Psi(M,\nu_0))\leq 1$.
Applying Lemma  A.1  in \cite{ChuMi} we deduce that for every $h\in {\mathcal A}_M$ and $\nu \in (0,{\nu}_0]$:
\begin{equation*}
\sup_N {\EX   }  \Big(\sup_{0\leq t\leq T} |\xi_{h}^{\nu}(t\wedge\bar{\tau}_{N})|_{H}^{2p }
+ \nu \int_0^{\bar{\tau}_{N}} |\nabla \xi_h^\nu (s) |_{H}^2 ds \Big)  \leq
C({\EX   } |\mbox{\rm curl } \zeta|_{H}^{2p}, M, T).
\end{equation*}
Since the previous upper bound is uniform in $N$, we deduce that $\bar{\tau}_N\to T$ as $N\to \infty$. Thus,
using the monotone convergence we get \eqref{xihnu_p},  which concludes the proof.
\end{proof}

The following well-posedeness result for problem  \eqref{uhnu} follows from
Propositions \ref{lemma1} and
\ref{lemma2}; the proof
is not given and  we refer to \cite{ChuMi} and \cite{SrSu} for details.

\begin{theorem}\label{theorem_existence_NSE}
Assume that $(\sigma_\nu, \nu >0)$ satisfies conditions {\bf (C1)} and  {\bf (C2)}
and $(\tilde{\sigma}_\nu, \nu>0)$ satisfies conditions {\bf (C1Bis)} and
{\bf (C2Bis)}. Let  $p\in [2,\infty)$ be such that
${\EX   }(\|\zeta\|^{2p})<\infty$.
Then for every $M>0$, $h\in {\mathcal A}_M$ and $\nu\in (0, \nu_0]$,  there exists a
unique weak solution $u_{h}^{\nu}$ in ${\mathcal C}([0,T];H)\cap L^2(0,T;V)$ of equation
\eqref{uhnu} with initial condition $u_{h}^{\nu}(0)=\zeta\in V$. Furthermore,  {\rm a.s.}
$u_{h}^{\nu}\in {\mathcal C} \big([0,T\big]; V\big) $  and
the inequalities \eqref{uhnu_estimate1}, \eqref{uhnu_estimate2}
and \eqref{uhnu_estimate3} hold.
\end{theorem}

\section{Well posedness of the inviscid problem}
The aim of this section is to deal with the inviscid case $\nu=0$, that is with the Euler  evolution equation
\begin{equation}\label{uh}
du_h^0(t) + B(u_h^0(t),u_h^0(t))\, dt =
\tilde{\sigma}_0 (t, u^0_h(t))\, h(t)\, dt\; , \quad u_h^0(0)=\zeta
\end{equation}
in  $[0,T]\times D$.

\begin{theorem}\label{existence_uh}
Let us assume that $\zeta\in V$ and that $\tilde{\sigma}_0$ satisfies conditions {\bf (C1Bis)} and  {\bf (C2Bis)}.
Then for   all $M>0$,  $h\in \mathcal{A}_{M}$ and  $T>0$, there exists {\rm a.s.} a solution
$u_{h}^{0}\in C\big([0,T]; H\big)\bigcap L^{\infty}\big(0,T;
V\big)$ for the equation \eqref{uh} with the initial condition
$u_{h}^{0}=\zeta$,  such that for all $ \varphi\in V$ and  $t\in [0,T]$
\begin{equation}\label{inviscid_weak}
\big( u_{h}^{0}(t), \varphi\big)  -\int_{0}^{t}
\langle B(u_h^0(s), \varphi),u_h^0(s) \rangle   ds=\int_{0}^{t}
\big( \tilde{\sigma_0}(s,u^0_h(s))\, h(s), \varphi\big) ds, \quad\mbox{\rm a.s.}
\end{equation}
Moreover, there exists a positive constant $C_{3}(M)$ (which also depends on $\tilde{K}_0$, $\tilde{K}_1$
and $T$)  such that for every $h\in {\mathcal A}_M$, one has a.s.
\begin{equation}\label{uh_estimate1}
\sup_{0\leq t\leq T}\|u_{h}^{0}(t)\|\leq C_{3}(M) (1+\|\zeta\|) .
\end{equation}
\end{theorem}

\begin{proof}
For $\mu>0$, let us approximate equation \eqref{uh} by the solution $u^{0\mu}_h$
to the following Navier Stokes  evolution equation:
\begin{equation}\label{uh_mu}
du_h^{0\mu}(t) +\big[ \mu Au_h^{0\mu}(t) +
B(u_h^{0\mu}(t),u_h^{0\mu}(t))\big] \, dt =
\tilde{\sigma}_0(t, u^{0\mu}_h(t))\, h(t)\, dt\; , \quad
u_h^{0\mu}(0)=\zeta,
\end{equation}
with the same incompressibility and boundary conditions.
If $\zeta\in H$, $\tilde{\sigma}_0$ satisfies the condition {\bf (C1Bis)}
and $h\in {\mathcal{A}}_{M}$ for $M>0$, then Proposition  \ref{lemma1} shows that {\rm a.s.}
equation \eqref{uh_mu} has a unique solution $u_h^{0\mu}\in
{\mathcal C}\big([0,T]; H\big)\bigcap L^{2}\big(0,T;V\big)$ (see also \cite{SrSu} or \cite{ChuMi}). Moreover, if
$\zeta\in V$ and $\tilde{\sigma}_0$ satisfies {\bf (C2Bis)}, then Proposition \ref{lemma2} implies that
{\rm a.s.}  $u_h^{0\mu}\in {\mathcal C} \big([0,T];V\big)$. In order to prove the
existence of solutions for equation  \eqref{uh}, we need some
estimates on $u_h^{0\mu}$ uniform  in $\mu>0$. Multiply the equation
\eqref{uh_mu} by $2 u_h^{0\mu}$ and integrate over $[0,t]\times D$; then
an argument similar to that used to prove Proposition \ref{lemma1}, based on  \eqref{Bnul},
the Cauchy-Schwarz and Young inequalities and
assumption {\bf (C1Bis)},  yields for every $\mu >0$:
\begin{align*}
|u_h^{0\mu}(t)|_{H}^{2} &+2 \mu\int_{0}^{t}\|u_h^{0\mu}(s)\|^{2}ds \leq
|\zeta|_{H}^{2}+
2 \int_{0}^{t}\big(\tilde{\sigma}_0(s,u^{0\mu}_h(s))h(s),u_h^{0\mu}(s) \big)ds\\
&\leq |\zeta|_{H}^{2}+ 2 \tilde{K}_0 \sqrt{ M T}  + 2\Big({\tilde{K}_0} + {\tilde{K}_1}\Big)
\int_0^t  |u_h^{0\mu}(s)|_{H}^2\, |h(s)|_0 ds .
\end{align*}
Hence, by Gronwall's lemma,  we deduce the existence of a constant $\tilde{C}_1$
which depends on $M,T, \tilde{K}_{0}$ and $\tilde{K}_{1}$ such that:
\begin{equation}\label{uhmu_estimate1}
\sup_{\mu>0} \, \sup_{0\leq t\leq T}|u_h^{0\mu}(t)|_{H}^{2}\leq
\tilde{C}_1 (1+|\zeta|_{H}^2) .
\end{equation}
Let $\xi_h^{0\mu}(t):=\mbox{\rm curl } u_h^{0\mu}(t)$;  then applying the
curl operator to equation \eqref{uh_mu}  and using \eqref{curlB} we obtain the following
evolution equation
\begin{equation}\label{xih_mu}
d\xi_h^{0\mu}(t) +\mu A\xi_h^{0\mu}(t) +
B(u_h^{0\mu}(t),\xi_h^{0\mu}(t))\, dt =
\mbox{\rm curl } \tilde{\sigma}_0(t, u^{0\mu}_h(t))\, h(t)\, dt\; ,
\end{equation}
with the initial condition $ \xi_h^{0\mu}(0)=\mbox{\rm curl } \zeta$.
Multiply the equation \eqref{xih_mu} by $2 \xi_h^{0\mu}$ and integrate
over $[0,T]\times D$ and use an argument similar to that in the proof of Proposition \ref{lemma2};
since  $\tilde{\sigma}_0$
satisfies the condition {\bf (C2Bis)}, using \eqref{curlBnul} for $q=2$, \eqref{normcurl},
Cauchy-Schwarz's and Young's inequalities, we deduce
\begin{align*}
&|\xi_h^{0\mu}(t)|_{H}^{2}  +2 \mu\! \int_{0}^{t}\!\!\! \|\xi_h^{0\mu}(s)\|^{2}ds
\leq
|\mbox{\rm curl }\zeta|_{H}^{2}+ 2 \! \int_{0}^{t}\!\!\!
|\mbox{\rm curl } \tilde{\sigma}_0(s, u^{0\mu}_h(s))|_{L(H_{0},H)}
|h(s)|_{0}  |\xi_h^{0\mu}(s)|_{H}  ds\\
&  \leq  \|\zeta\|^{2}   + \Big( 2 \tilde{K}_{0}
+ \tilde{K}_1 \!\! \sup_{s\in [0,T]} |u^{0\mu}_h(s)|_{H}^2\Big)  \sqrt{MT}
+  2 \! \int_0^t \!\! \left( {\tilde{K}_0} + (2 C+1)  {\tilde{K}_{1}}  \right) |h(s)|_{0} |\xi_h^{0\mu}(s)|_{H}^{2}ds.
\end{align*}
Thus, \eqref{uhmu_estimate1} and Gronwall's lemma yield the existence of a constant $\tilde{C}_2:= \tilde{C}_2(M,T,
\tilde{K}_{0},\tilde{K}_{1})$
such that  for every $h\in {\mathcal A}_M$:
$\sup_{\mu>0} \,  \sup_{0\leq t\leq T}|\xi_h^{0\mu}(t)|_{H}^{2}\leq
\tilde{C}_2 (1+ \| \zeta\|^2) \; \mbox{\rm a.s.}$
Combining this  estimate,  \eqref{uhmu_estimate1}
and \eqref{normcurl},
we deduce the existence of a constant $\tilde{C}_3$ depending on
$M,T, \tilde{K}_{0}$ and $\tilde{K}_{1}$
such that
for any
$h\in {\mathcal A}_M$  one has:
\begin{equation}\label{uhmu_estimate2}
\sup_{\mu >0}   \sup_{0\leq t\leq T}\|u_{h}^{0\mu}\|\leq
\tilde{C}_3 (1+ \| \zeta \|)
\ \ {\rm a.s.}
\end{equation}
Furthermore,  we have $u_{h}^{0\mu}\in {\mathcal C}\big([0,T]; H\big)\bigcap L^{\infty}\big(0,T; V\big)$  a.s. for every $\mu>0$,  and
$$u_h^{0\mu}(t) =\zeta-\mu\int_{0}^{t} Au_h^{0\mu}(s) -
\int_{0}^{t}B(u_h^{0\mu}(s),u_h^{0\mu}(s))\, ds +
\int_{0}^{t}\tilde{\sigma}_0(s,u^{0\mu}_h(s))\, h(s)\, ds.$$
Using the estimates \eqref{uhmu_estimate1}, \eqref{uhmu_estimate2},
assumptions {\bf (C2Bis)} on  $\tilde{\sigma}_0$ and \eqref{normBrq} for $q=2$ and $r=1$,
we deduce the existence of a constant $\tilde{C}_4$
  depending on
$M,T, \tilde{K}_{0}$ and $\tilde{K}_{1}$
such that  the following  estimate holds
for any $\mu\in (0,1]$ and $h\in {\mathcal A}_M$:
\begin{equation}\label{uhmu_estimate3}
\|u_h^{0\mu}\|_{W^{1,2}(0,T;V')}\leq \tilde{C}_{4} (1+\| \zeta\|).\ \ {\rm a.s.}
\end{equation}
By classical compactness arguments, we can  extract a subsequence
(still denoted $u_h^{0\mu}$) and prove the existence of  a function
$v\in
W^{1,2}\big(0,T;V'\big)
\bigcap L^{\infty}\big(0,T; V \big)$ such that as $\mu \to 0$:
\begin{align*}
u_h^{0\mu} & \to  v\ \mbox{ \rm weakly\  in }\ L^{2}\big(0,T; V\big)\; \mbox{\rm and in } W^{1,2}\big(0,T; V'\big),\\
u_h^{0\mu} &\to  v\ \mbox{ \rm strongly\  in}\ L^{2}\big(0,T; H\big)\, ,
\\
u_h^{0\mu} &\to  v\ \mbox{ \rm   in the weak star topology of }\ L^{\infty}\big(0,T; V\big)\, .
\end{align*}
Letting  $\mu\rightarrow 0$  in  equation \eqref{uh_mu},
we deduce  that the above limit $v$ is solution of the equation \eqref{uh},
that is  $v=u_{h}^{0}$. Moreover,
\eqref{uhmu_estimate2}
being uniform in $\mu>0$,  we deduce
\eqref{uh_estimate1}.
\end{proof}
Uniqueness of the solution to the Euler equation is known to be a more difficult problem
and the classical  deterministic results use non-Hilbert Sobolev spaces $H^{1,q}$
for  $q\in [2,+\infty)$. This requires to impose some $H^{1,q}$-control on the coefficient $\tilde{\s}_0$,
which is stated below in Condition  \textbf{(C3qBis)} for $\nu=0$.
It will enable
us to prove the uniqueness of the solution to the "deterministic"
inviscid equation in $H^{1,q}$ when $2\leq q<\infty$ in the next result. The following general
assumptions \textbf{(C3q)} and \textbf{(C3qBis)} on $\s_\nu$ and $\tilde{\s}_\nu$  will also
yield some apriori estimates
for the $q$-th moment of
the $H^{1,q}$-norm of the solution to the stochastic controlled
equation which will be proven in section \ref{aprioriH1q}.
This will be needed in order to prove the large deviations result as $\nu \to 0$ in section \ref{s4}.

\noindent \textbf{Condition (C3q):} {\it Let $q\in [2,\infty)$;
$ \s_\nu  \in {\mathcal C}\big([0,T]\times H^{2,q}; R(H_0, H^{1,q})\big)$ for $\nu >0$,
there exist non negative  constants $K_i$ such that
for every $u\in H\bigcap H^{2,q}$ and $\nu >0$, } if $\xi = \mbox{\rm curl }u$,\\
$\| \mbox{\rm curl } \s_\nu (t,u)\|_{R(H_0, L^q)}^2
\leq K_3+ K_4\|u\|_q^2 + K_5 \|\xi \|_q^2$.

\noindent \textbf{Condition  (C3qBis):} {\it  Let $q\in [2,\infty)$;
$\tilde{\s}_\nu \in {\mathcal C}\big([0,T]\times H^{1,q} ;
L(H_0, H^{1,q})\big)$ for $\nu \geq 0$, and there exist non negative  constants
$\tilde{K}_i$ such that for every $u\in H^{1,q} $ and $\nu >0$   (resp.
$u\in H^{2,q}$ for $\nu =0$) if  $\xi = \mbox{\rm curl } u$,\\
$\| \mbox{\rm curl } \tilde{\s}_\nu(t,u)\|_{L(H_{0},L^q )}
\leq \tilde{K}_3+  \tilde{K}_4\|u\|_{q}  + \tilde{K}_5\|\xi \|_{q}$.
}

\par
The following theorem shows that if $\mbox{\rm curl } \zeta$ is bounded,
then the solution to  \eqref{uh} is unique. This will be a key ingredient of
the identification for the rate function of tle LDP in  section \ref{s4}.
\begin{theorem}\label{uniquness_uh}
Let us assume that the assumptions of Theorem \ref{existence_uh} are
satisfied. Moreover, let us assume that $\mbox{\rm curl } \zeta\in
\big(L^{\infty}(D)\big)^{2}$ and that condition {\bf (C3qBis)} holds
for every  $q\in [2,\infty)$ and $\nu=0$. Then, for every $M > 0$ and $h\in {\mathcal A}_M$,
the solution of equation \eqref{uh} with the
initial condition $u_{h}^{0}=\zeta$ is {\rm a.s.} unique in
${\mathcal C}\big([0,T]; H\big)\bigcap L^{\infty}\big(0,T; H^{1,q}
\big)$ for every  $q\in [ 2,\infty) $ and  every $T>0$.
Moreover, there exist  positive constants  $C_{4}(M)$ and $\bar{C}_4(M)$
(which also depend on $T$,
$\tilde{K}_i$ and $\|\zeta\|_{L^\infty(D)^2})$, such that
for every $h\in {\mathcal A}_M$ and $q\in [2,\infty)$, one has a.s.
\begin{align}
\sup_{0\leq t\leq T}  \| \mbox{\rm curl } u_{h}^{0}(t\|_q &\leq
    C_{4}(M) (1+\|\zeta\| + \| \mbox{\rm curl }\zeta\|_q ) ,
\label{uh_estimate2}\\
\sup_{0\leq t\leq T} \|\nabla u_h^0(t)\|_q & \leq
\bar{C}_4(M) q (1+\|\zeta\| + \| \mbox{\rm curl }\zeta\|_q ) .
\label{uh_gradiantq}
\end{align}
\end{theorem}
\begin{proof}
The first step of the proof will  establish the estimates
\eqref{uh_estimate2} and \eqref{uh_gradiantq}. The second step will  prove the
uniqueness of the solution $u_{h}^0$. \\
{\bf Step1.}  (Existence)
Using \eqref{normcurl} one sees that
the proof of \eqref{uh_gradiantq} reduces to that of
\eqref{uh_estimate2}, that is  to check $L^q(D)$ upper bounds for
$\xi_{h}^{0}(t):=\mbox{\rm curl } u_h^0(t)$.
Replacing  $u^0_h$ by its Galerkin approximation $u^0_{h,n}$, we may assume that $u^0_{h,n} \in H^{2,q}$
and deduce the desired inequality by proving upper bounds which do not depend on $n$. To ease notations
in the sequel, we skip the index $n$.

Let us apply the curl to the equation \eqref{uh}; the identity \eqref{curlB} yields
\begin{equation}\label{uh_curl}
d\xi_{h}^{0}(t) + B\big( u_h^0(t), \xi_h^0(t) \big)  \, dt =
\mbox{\rm curl } \tilde{\sigma}_0 (t,u^0_h(t))\, h(t)\, dt\; , \quad
\xi_{h}^{0}(0)=\mbox{\rm curl } \zeta.
\end{equation}
Let us multiply the equation \eqref{uh_curl} by
$ q |\xi_{h}^{0}(t)|^{q-2}\xi_{h}^{0}(t)$ and integrate over  $[0,t]\times D$; we obtain
\begin{align*}
\|\xi_{h}^{0}(t)\|_{q}^{q}
& + q \int_0^t\!\!  \int_{D} \!     (u_h^0(s)\cdot\nabla)\xi_h^0(s)
|\xi_{h}^{0}(s)|^{q-2}\xi_{h}^{0}(s) dx ds  =
\| \mbox{\rm curl }\zeta \|_q^q \\
&\qquad +  q  \int_0^t \!\! \int_{D} \mbox{\rm curl }  \tilde{\sigma}_0(s,u^0_h(s))
h(s)|\xi_{h}^{0}(s)|^{q-2}\xi_{h}^{0}(s) dx ds.
\end{align*}
Since $\xi_h^0(t)=\mbox{\rm curl }u_h^0(t)$,  \eqref{curlBnul} implies
$\int_{D}(u_h^0(s) \cdot\nabla)\xi_h^0(s)
|\xi_{h}^{0}(s)|^{p-2}\xi_{h}^{0}(s) dx =0$  for every $s$.
On the other side,  the H\"{o}lder and Young inequalities and
{\bf (C3qBis)}  yield:
\begin{align*}
\|\xi_{h}^{0}(t)\|_{q}^{q} & \leq  \| \mbox{\rm curl }\zeta \|_q^q +
q \int_0^t \big( \tilde{K}_3 + \tilde{K}_4 \|u^0_h(s)\|_q  +
\tilde{K}_5 \|\xi^0_h(s)\|_q\big)
|h(s)|_0 \|\xi^0_h(s)\|_q^{q-1} ds \\
&\leq   \| \mbox{\rm curl }\zeta \|_q^q + \Big( q\tilde{K}_3
+  \tilde{K}_4  \sup_{0\leq s\leq T}  \|u^0_h(s)\|^q_q\Big)   \sqrt{MT} \\
&\qquad +  q \big( \tilde{K}_3 + \tilde{K}_4   + \tilde{K}_5
\big)\int_0^t |h(s)|_0  \|\xi_h^0(s)\|_q^q  ds.
\end{align*}
Finally, the inclusion $V=H^{1,2}\subset L^q(D)$ given by \eqref{Sobolevq},
the control of the $V$
norm proven in \eqref{uh_estimate1} (which clearly also holds for the Galerkin approximation $u^0_{h,n}$
with an upper bound which does not depend on $n$) and  Gronwall's lemma imply the existence of
a non negative constant $\tilde{C}_5(M)$,
depending on $T, M, \tilde{K}_i $
such that for any $n\geq 1$ and
$h\in {\mathcal A}_M$, we have
\begin{align}   \sup_{0\leq t\leq T}&  \|\xi_{h,n}^{0}(t)\|_{q}^{q}  \leq
   \Big(\| \mbox{\rm curl } \zeta \|_q^q
+ \Big[ q \tilde{K}_3   + \tilde{K}_4 \, \sup_{0\leq t\leq T} \|u^0_h(t)\|_q^q \Big] \Big) \;
e^{q \tilde{C}_5(M)} \label{majocurlq_1}\\
&\leq   \Big( \|  \mbox{\rm curl } \zeta \|_q^q +  \big[q  \tilde{K}_3  + \tilde{K}_4 C(q)^q C_3(M)^q \,
2^{q-1}(1+\|\zeta\|^q)\big]\Big)  \, e^{q \tilde{C}_5(M)} \nonumber .
\end{align}
Since $\sup\{q^{\frac{1}{q}} : 2\leq q<\infty\}<\infty$,  if $C(q)$ denotes the constant in \eqref{Sobolevq},
as  $n\to \infty$ classical
arguments conclude that $\sup_{0\leq t\leq T} \|\xi_{h,n}^0(t)\|_q
\leq \| \mbox{\rm curl } \zeta\|_q +  C(T,M) (1+ C(q)) \big(1+\| \zeta\|\big) $ for every
$h\in {\mathcal A}_M$ and $n\geq 1$. Using \eqref{normcurl} for some
$q_0\in [2,q)$, we deduce that $\sup_{0\leq t\leq T} \|\nabla
u^0_{h,n}(t)\|_{q_0} \leq C(T,M,q_0) \big( 1+ \| \mbox{\rm curl } \zeta \|_{q_0}  + \|\zeta\|\big) $ a.s.
Thus, the Sobolev embedding \eqref{Sobolev} yields the existence of a
constant  $\tilde{C}_6(M,T)$,  such that  $\sup_{0\leq t\leq T} \|u_{h,n}^0(t)\|_{L^\infty (D)}
\leq  \tilde{C}_6(M,T)  \big( 1+ \| \mbox{\rm curl } \zeta \|_{q_0}
+ \|\zeta\|\big)  $ a.s.  for any $n\geq 1$ and
$h\in {\mathcal A}_M$.
Since $D$ is bounded, using  this inequality in \eqref{majocurlq_1},  we deduce
that
$\sup_{0\leq t\leq T}  \|\xi_{h,n}^{0}(t)\|_{q}^{q}
\leq  \exp(q \tilde{C}_5(M))   \big[  \| \mbox{\rm curl } \zeta \|_q^q + 3^q  \tilde{C}_6(M,T)^q   \big(1+ \|\zeta \|^q +
\| \mbox{\rm curl } \zeta \|_{q_0}^q \big)  \big] $
a.s.
for every integer $n\geq 1$ and any $h\in {\mathcal A}_M$.
Since $ q\in [q_0,+\infty)$ and $D$ is bounded,   we deduce  $ \| \mbox{\rm curl } \zeta \|_{q_0}^q \leq [\lambda(D)\vee 1]^q \| \mbox{\rm curl } \zeta \|_{q}^q$.
Letting $n\to \infty$ and using classical arguments, we conclude the proof of   \eqref{uh_gradiantq}.

\noindent {\bf Step 2.} (Uniqueness)
Let us mention  that the proof of the uniqueness is based on
\cite{Bardos} and \cite{Yu} adapted to the nonhomogeneous random case.
Using the estimate \eqref{uh_gradiantq}
for  some $q>2$ (such as $q=4$) and \eqref{Sobolev}, we deduce that any solution $u_h^0$ to  \eqref{uh}
belongs to  $L^{\infty}((0,T)\times D)$. Let
$u_{h}^{0}$ and $v_{h}^{0}$ be two solutions for equation \eqref{uh}
with the same initial condition and let us denote by
$z:=u_{h}^{0}-v_{h}^{0}$; then  $z$ is solution of $z(0)=0$ and
\begin{equation*}
dz(s) + \left[B(u_{h}^{0}(s),u_{h}^{0}(s))-B(v_{h}^{0}(s),v_{h}^{0}(s))\right]ds
=\left[\tilde{\sigma}_0(s,u^0_h(s))-\tilde{\sigma}_0(s,v^0_h(s))\right]h(s) ds .
\end{equation*}
Let us multiply the above equation by $ z(t)$ and integrate on $D$, use assumption {\bf
(C1Bis)} on $\tilde{\sigma}_0$, the Schwarz and H\"older inequalities and \eqref{uh_gradiantq}. This yields for
any $q\in (1,\infty)$,  when $q^*= \frac{q}{q-1}$ denotes the conjugate exponent of $q$:

\begin{align*}
\frac{1}{2} \frac{d}{dt}&|z(t)|_H^{2}= - (B(z(t),u_{h}^{0}(t)), z(t))
+   \big( \left[ \tilde{\sigma}_0(t,u^0_h(t))-\tilde{\sigma}_0(t,v^0_h(t))\right] h(t) , z(t) \big) \\
&\leq  \int_{D}|z(t)|^{2}(x) |\nabla u_{h}^{0}(t)|(x) \, dx
+|\left(\tilde{\sigma}_0(t,u^0_h(t))-\tilde{\sigma}_0(t,v^0_h(t))\right)|_{L(H_0,H)} |h(t)|_0 \, |z(t)|_H\\
&\leq   \|\nabla u_h^0(t) \|_q \|z(t)\|_{L^\infty(D)}^{\frac{2}{q}} |z(t)|_H^{\frac{2}{q^*}}
+ \tilde{L}_{1}|u^0_h(t)-v^0_h(t)|_H |h(t)|_{0}|z(t)|_H .
\end{align*}
Set
$Z :=\sup_{0\leq t\leq T} \|z(t)\|_{L^{\infty}(D)}$ and $X(t):=|z(t)|_{H}^{2}$. Since $D$ is bounded,
there exists a constant $C\geq 1$ such that $\|\mbox{\rm curl }\zeta\|_q \leq C  \|\mbox{\rm curl }\zeta\|_\infty$ for
every  $q\in [2,\infty) $;   then
$X(0)=0$ and  for $t\in [0,T]$, \eqref{uh_gradiantq} yields
\begin{equation*}
X'(t) \leq  2 C q \, \bar{C}_4(M) [1+ \|\zeta\|
+ \|\mbox{\rm curl } \zeta \|_{L^{\infty}(D)}]   Z^{\frac{2}{q}} X(t)^{1-\frac{1}{q}} +
2 \tilde{L}_1 |h(t)|_0 X(t),
\end{equation*}
which leads to
\begin{equation*}
\int_{0}^{t}\frac{X'(s)}{X(s)^{1-1/q}} ds \leq
2 C q\,  \bar{C}_4(M) [1+ \|\zeta\|
+ \|\mbox{\rm curl } \zeta \|_{L^{\infty}(D)}]  Z^{\frac{2}{q}} t
+ \int_{0}^{t} 2 \tilde{L}_{1}|h(s)|_{0}  X(s)^{1/q} ds.
\end{equation*}
Hence, using Gronwall's lemma, we deduce that for $q\in [2,\infty)$ and $t\in [0,T]$,
\begin{eqnarray*}
X(t)^{\frac{1}{q}}&\leq&
2 \bar{C}_4(M) [1+ \|\zeta\|
+ \|\mbox{\rm curl } \zeta \|_{L^{\infty}(D)}]   Z^{\frac{2}{q}} t
+\frac{2}{q} \int_{0}^{t}\tilde{L}_{1} |h(s)|_{0} X(s)^{\frac{1}{q}} ds\\
&\leq& 2 \bar{C}_4(M) [1+ \|\zeta\|
+ \|\mbox{\rm curl } \zeta \|_{L^{\infty}(D)}]   Z^{\frac{2}{q}} t \exp( \tilde{L}_1 \sqrt{MT}).
\end{eqnarray*}
Finally,  we get the following estimate for any $T^* \in [0,T]$
and $q\in (2,\infty)$:
\begin{equation}\label{uniqueness_estimate}
\sup_{0\leq t\leq T^*} |z(t)|_{H}^{2}\leq\left(2\bar{C}_4(M) [1+ \|\zeta\|
+ \|\mbox{\rm curl } \zeta \|_{L^{\infty}(D)}] T^* \exp(2\tilde{L}_1 \sqrt{MT})  \right)^q Z^2.
\end{equation}
Thus, choosing $T_1^{*}>0 $ small enough and letting
$q\longrightarrow\infty$, we deduce that $|z(t)|_{H}^{2}=0$ for every
$t \in [0,  T_1^{*}]$. Repeating this argument with $u_h^0(T_1^*)= v_h^0(T_1^* ) $
instead of $\zeta$ and using \eqref{Sobolev}, \eqref{uh_gradiantq} and \eqref{uh_estimate1},  we conclude
that there exists $T^*>0$ such that $|z(t)|_{H}^{2}=0$ for every integer $k=0,1, \cdots$ and any
$t\in [T_1^{*} + k T^*,
T_1^{*} + (k+1) T^*]\cap [0,T]$. This concludes the proof
of the uniqueness.
\end{proof}

\section{Apriori bounds of the stochastic controlled equation in $H^{1,q}$}\label{aprioriH1q}
In order to prove the large deviation principle for the solution $u$ to \eqref{Navier-Stokes}, we need to
obtain  more regularity  and   apriori bounds for the solution $u^\nu_h$ to the stochastic controlled
equation \eqref{uhnu} in
the Sobolev spaces $H^{1,q}$ for $q\in [2,+\infty)$.
This requires some more conditions on the diffusion coefficient $\sigma_\nu$ and $\tilde{\s}_{\nu}$ introduced
in the previous section. It also relies on  the stochastic
calculus in Banach spaces, which is  briefly described in the subsection \ref{Radon} of the Appendix.

\begin{Prop} \label{prop4.1}
Suppose that $\EE|\zeta|_H^{2p}<\infty$ for some $p\in [2,\infty)$  and let $q\in [2,\infty)$
be such that $E\|\zeta\|_{H^{1,q}}^q <\infty$. Assume that
$\sigma_\nu$ satisfies conditions {\bf (C1)}--{\bf (C3q)} and $\tilde{\sigma}_\nu$ satisfies conditions
{\bf (C1Bis)}--{\bf (C3qBis)}.
Then for  $M>0$, $\nu_0>0$,  $h\in {\mathcal A}_M$
and  $\nu \in (0,{\nu}_0]$,  the solution $u_h^\nu$ to \eqref{uhnu} belongs to
$L^\infty(0,T ; H^{1,q})$ a.s.  Furthermore, there exists a constant $C_5(M ,q)$ such that
\begin{equation}\label{uhnu_H1q}
\sup_{0<\nu\leq {\nu}_0}\, \sup_{h\in {\mathcal A}_M} {\EX} \Big( \sup_{0\leq t\leq T}
\|u^\nu_h(t)\|_{H^{1,q}}^q \Big) \leq
C_5(M,q)\, \big( 1+{\EX   } \|\zeta \|_{H^{1,q}}^q \big) .
\end{equation}
\end{Prop}
\begin{proof} The Sobolev embedding inequality \eqref{Sobolevq} and Proposition \ref{lemma2}
imply that for $0<\nu\leq {\nu}_0$
and $h\in {\mathcal A}_M$, ${\EX   }\big( \sup_{0\leq t\leq T} \|u^\nu_h(t)\|_q^q \big)
\leq C(q)^q  C_2(q,M) \big(1+{\EX   } \|\zeta\|^q\big)$.  Using the inequality \eqref{normcurl}, one sees that the proof of
\eqref{uhnu_H1q} reduces to check that if $\zeta_h^\nu=\mbox{\rm curl }u^\nu_h$,
\begin{equation}\label{curl_H1q}
\sup_{0<\nu\leq {\nu}_0}\,
\sup_{h\in {\mathcal A}_M} {\EX   } \Big( \sup_{0\leq t\leq T} \|\xi^\nu_h(t)\|_{q}^q \Big)
\leq \bar{C}_6(M,q)\, \big( 1+{\EX   } \|\mbox{\rm curl }\zeta \|_{q}^q \big).
\end{equation}
We use once more the Galerkin approximation $u^\nu_{h,n}$ of $u^\nu_h$ and prove an estimate similar to \eqref{curl_H1q}
for $\xi^\nu_{h,n} = \mbox{\rm curl }u^\nu_{h,n}$ with a constant $C_6(M,q)$ which does not depend on $n$.
The process  $\xi^\nu_{h,n}$ satisfies an equation similar to \eqref{xihnu} and once more
to ease notations, we will skip the index $n$.
Let $\langle .,.\rangle$ denote the duality between  $L^q(D)$ and $L^{q^*}(D)$
for some $q^*=\frac{q}{q-1}$.
For fixed $N>0$, let $\tau_N=\inf\{ t\geq 0 : \|\xi_{h}^\nu(t)\|_q \geq N\}\wedge T$.
The It\^o formula \eqref{ItoRadon} and the upper estimate \eqref{Itocor} yield
\begin{equation} \label{decH1p}
\|\xi^\nu_h(t\wedge \tau_N)\|_q^q  \leq  \|\mbox{curl }\zeta\|_q^q + J(t) + \sum_{1\leq i\leq 4} T_i(t),
\end{equation}
where we have:
\begin{eqnarray*}
J(t)&=& q\sqrt{\nu} \int_0^{t\wedge \tau_N} \langle |\xi_h^\nu(s)|^{q-2} \xi_h^\nu(s)\, ,\,
\mbox{\rm curl }\sigma_\nu(s,u^\nu_h(s))
dW(s)\rangle ,\\
T_1(t)&=& - q \nu \int_0^{t\wedge \tau_N}  \langle |\xi_h^\nu(s)|^{q-2} \xi_h^\nu(s)\, ,\, A\xi_h^\nu(s) \rangle ds,\\
T_2(t)&=& -q \int_0^{t\wedge \tau_N}  \langle |\xi_h^\nu(s)|^{q-2} \xi_h^\nu(s)\, ,\,
\mbox{\rm curl }B(u_h^\nu(s) , u_h^\nu(s))
\rangle ds, \\
T_3(t)&=& q\int_0^{t\wedge \tau_N}  \langle |\xi_h^\nu(s)|^{q-2} \xi_h^\nu(s)\, ,\,
{\rm curl \, } \tilde{\sigma}_\nu(s,u^\nu_h(s))
h(s) \rangle ds,\\
T_4(t)&=& \frac{q}{2}(q-1) \nu \int_0^{t\wedge \tau_N} \|\mbox{curl } \sigma_\nu(s,u^\nu_h(s))\|_{R(H_0,L^q)}^2
\|\xi_h^\nu(s)\|^{q-2}_q\, ds.
\end{eqnarray*}
Since $\xi_{h}^\nu=0$ on $\partial D$ and $A=-\Delta$,  we have:
\begin{eqnarray*}
T_1(t)&=&-q \nu \int_0^{t\wedge \tau_N} \!\! ds\int_D \langle \nabla\big( |\xi_h^\nu(s)|^{q-2} \xi_h^\nu(s)\big) \, , \,
\nabla \xi^\nu_h(s)\rangle
\, dx \\
&=& -q (q-1)\nu  \int_0^{t\wedge \tau_N} \!\! ds\int_D |\xi_h^\nu(s)|^{q-2} \, |\nabla \xi_h^\nu(s)|^2 dx .
\end{eqnarray*}
Since $\xi_h^\nu = \mbox{\rm curl } u_h^\nu$, equation \eqref{curlBnul} implies that $T_2(t)=0$.
H\"older's and Young's inequalities and the assumption {\bf (C3qBis)}
yield:  
\begin{align*}
T_3&(t)  \leq  q\int_0^{t\wedge \tau_N} \| |\xi_h^\nu(s)|^{q-1} \|_{q^*} \, \|\mbox{\rm curl } \tilde{\sigma}_\nu(s,u^\nu_h(s))
h(s)\|_q\, ds\\
&\leq  q\tilde{K}_3 \sqrt{MT}  + \tilde{K}_4 \int_0^{t\wedge \tau_N} \!\! \|u^\nu_h(s)\|_q^q \, |h(s)|_0 \, ds
+ \int_0^{t\wedge \tau_N} \!\! \!\! \|\xi_h^\nu(s)\|_q^q   q\big[ \tilde{K}_3 + \tilde{K}_4 +
  \tilde{K}_5 \big] |h(s)|_0
ds.
\end{align*}
Condition {\bf (C3q)}, H\"older's and Young's inequalities imply that for any
$\nu \in (0,\nu_{0}]$,
\begin{align*}
T_4(t) &\leq   \frac{q(q-1)}{2} \nu \int_0^{t\wedge \tau_N} \!\! \|\xi_h^\nu(s)\|_q^{q-2} \Big[K_3+K_4\|u^\nu_h(s)\|_q^2 +
K_5 \|\xi^\nu_h(s)\|_q^2 \Big] ds\\
&\leq  \frac{ q(q-1)}{2} {\nu} K_3 T + \nu (q-1) K_4 \int_0^{t\wedge \tau_N} \!\!\! \|u^\nu_h(s)\|_q^q ds
\\ & \qquad
+ \nu \frac{q-1}{2}  \int_0^{t\wedge \tau_N}\!\! \big[ q (K_3+K_5) +  K_4 (q-2) \big]  \|\xi_h^\nu(s)\|_q^q  ds.
\end{align*}
For $t\in [0,T]$, let
\[ X(t)=\sup_{0\leq s\leq t} \|\xi_h^\nu(s\wedge \tau_N)\|_q^q\quad \mbox{\rm and } \quad
Y(t)=  \int_0^{t\wedge \tau_N} ds \int_D |\xi_h^\nu(s)|^{q-2} |\nabla \xi_h^\nu(s)|^2 dx.\]
Then for any $\lambda \in (0,1)$, the
inequality \eqref{decH1p} and the above estimates of $T_i(t)$ imply that
\[ \lambda X(t) + (1-\lambda) \|\xi_h^\nu(t\wedge \tau_N)\|_q^q + \nu q(q-1)(1-\lambda)   Y(t)
\leq Z + \int_0^t \!\! \varphi(s) X(s) ds + I(t),
\]
where
\begin{align*}
&I(t)= \sup_{0\leq s\leq t} J(s),\\
&Z= \|\mbox{\rm curl }\zeta \|_q^q + q\tilde{K}_3 \sqrt{MT} + \frac{q(q-1)}{2} \nu K_3 T
+ \int_0^{t\wedge \tau_N}\!\!\!  \big[ \nu K_4 (q-1) + \tilde{K}_4 |h(s)|_0\big] \|u^\nu_h(s)\|_q^q ds,\\
&\varphi(s) =  q (\tilde{K}_3 + \tilde{K}_4+\tilde{K}_5 ) |h(s)|_0
+ \frac{q-1}{2}  \nu q (K_3+K_4+K_5)   .
\end{align*}
Set $\lambda = \frac{1}{2}$; then there exists a constant $\Phi(\nu_{0},M)$
such that for  $\nu \in (0,\nu_{0}] $ and $h\in {\mathcal A}_{M}$, almost surely
let $\int_0^T \varphi(s) ds \leq  \Phi(\nu_{0},M)$ and
\[ \frac{1}{2} X(t) + \frac{\nu}{2} q(q-1) Y(t) \leq Z + \int_0^t \varphi(s) X(s) ds + I(t). \]
Furthermore, using the Burkholder-Davies-Gundy inequality \eqref{BDGRadon}, condition {\bf (C3q)}, H\"older's and Young's
inequalities,  we deduce that for any $\beta >0$,
\begin{align*}
& {\EX   } I(t)
\leq \sqrt{\nu} C_1 q {\bf E} \Big( \int_0^{t\wedge \tau_N} \|\mbox{\rm curl } \sigma_\nu(s,u^\nu_h(s))\|_{R(H_0,L_q)}^2
\|\xi_h^\nu(s)\|_q^{2(q-1)} ds\Big)^{\frac{1}{2}}\\
& \leq \sqrt{\nu} C_1 q {\EX   } \Big(\sup_{0\leq s\leq t} \|\xi_h^\nu(s\wedge \tau_N)\|^{\frac{q}{2}}_q
\Big[ \int_0^{t\wedge \tau_N}\!\!\!  \|\xi_h^\nu(s)\|_q^{q-2} \Big\{  K_3 + K_4 \|u_h^\nu(s)\|_q^2
+  K_5 \|\xi_h^\nu(s)\|_q^2
\Big\}
ds \Big]^{\frac{1}{2}} \Big)\\
&\leq \beta {\EX   } X(t) + \gamma {\bf E} \int_0^t X(s) ds + \bar{Z},
\end{align*}
where
\[  
\gamma=  \frac{1}{4\beta} \nu C_1^2 \big[ q^2 (K_3+K_5) +  q K_4 (q-2)
\big]  ,\quad
\bar{Z}= \frac{K_4}{2\beta} \nu q C_1^2 {\EX   }\int_0^{t\wedge \tau_N} \!\!\! \|u_h^\nu(s)\|_q^q ds +  \frac{K_3 T \nu q^2}{4\beta} .
\]  
Set $\alpha = \frac{\nu}{2} q(q-1)$, and choose $\beta >0$ such that
$2\beta e^{\Phi(\nu_{0},M) } \leq 1/2$.
Then, using once more Lemma A1 in \cite{ChuMi},
we conclude that \eqref{curl_H1q} holds for the Galerkin approximation $\xi_{h,n}^\nu$ of $\xi_h^\nu$ with a constant
$C_6(M,q)$ which does not depend on $n$.  A classical weak convergence argument
concludes the proof.
\end{proof}

\section{Large deviations}  \label{s4}
We will prove a large deviation principle using  a weak convergence approach
\cite{BD00, BD07}, based on variational representations of
infinite dimensional Wiener processes. 
For every $\nu >0$,
let 
$\sigma_\nu=\tilde{\sigma}_\nu$  satisfy the conditions  {\bf (C1)}, {\bf (C2)} and ${\bf (C3q)}$ for
every $q\in [2,+\infty)$.
Furthemore, we assume that the following condition holds:\\
 {\bf Condition (C4)}: There exists $\sigma_0$ satisfying conditions
{\bf (C1)}, {\bf (C2)} and ${\bf (C3q)}$ for
every $q\in [2,+\infty)$, such that for some map $\nu \in (0,+\infty) \to C(\nu)\in [0,+\infty)$ which
converges to 0 as $\nu \to 0$,  the upper estimate
\begin{equation}\label{conv_sigma}
\sup_{0\leq t\leq T} \big| \sigma_\nu(t,u) - \sigma_0(t,u)\big|_{L(H_0,H)} \leq C(\nu) \big[1+|u|_H\big]
\end{equation}
holds for $u\in H$ and $\nu >0$.

Note that as in the case of Hilbert-Schmidt operators with Hilbert spaces,
we have  $\|\Phi\|_{L(H_0,L^q)} \leq C \|\Phi\|_{R(H_0, L^q)}$.  
Then, for $\nu \geq 0 $,
the coefficients $\s_\nu$ 
also
satisfy the conditions   {\bf (C1Bis)}-{\bf (C3qBis)}    with appropriate coefficients.

Let $\mathcal{B}$ denote the  Borel $\s-$field of the Polish space
\begin{equation}\label{defX}
{\mathcal X}= C\big([0,T]; H\big)\bigcap L^{\infty}\big(0,T;
H^{1,q} \cap V\big) \bigcap L^{2}\big(0,T; \mathcal{H}\big)
\end{equation}
endowed with the norm
$\|u\|_{\mathcal X}:=\left(\int_{0}^{T}\|u(t)\|_{\mathcal H}^2 dt\right)^{1/2} $
and
\begin{equation}\label{defY}
{\mathcal Y}=\displaystyle\left\{\zeta\in V,\ {\rm such}\ {\rm that}\ \mbox{\rm curl }\zeta\in L^{\infty}(D)\right\}
\end{equation}
endowed with the norm $\|.\|_{\mathcal Y}$ defined by:
$$\|\zeta\|_{\mathcal Y}^2
:=\|\zeta\|^2 + \|\mbox{\rm curl }\zeta\|_{L^{\infty}}^2.$$

Note that using \eqref{normcurl} and \eqref{Sobolevq} we deduce that
${\mathcal Y}\subset H^{1,q}$ for any $q\in [2,\infty)$.
We will establish a LDP in the set ${\mathcal X}$ for the family of distributions of the solutions
$u^\nu= {\mathcal G}_\zeta^\nu(\sqrt{\nu}W)$ to the evolution equation \eqref{SNS} with initial condition $u^\nu(0)=\zeta \in \mathcal{Y}$.
\begin{definition}
The random family
$(u^\nu )$ is said to satisfy a large deviation principle on
${\mathcal X}$  with the good rate function $I$ if the following conditions hold:\\
\indent \textbf{$I$ is a good rate function.} The function  $I: {\mathcal X} \to [0, \infty]$ is
such that for each $M\in [0,\infty[$ the level set $\{\phi \in {\mathcal X}: I(\phi) \leq M
\}$ is a    compact subset of ${\mathcal X}$. \\
For $A\in \mathcal{B}$, set $I(A)=\inf_{u \in A} I(u)$.\\
\indent  \textbf{Large deviation upper bound.} For each closed subset
$F$ of ${\mathcal X} $:
\[
\lim\sup_{\nu\to 0}\; \nu \log \PX(u^\nu \in F) \leq -I(F).
\]
\indent  \textbf{Large deviation lower bound.} For each open subset $G$
of ${\mathcal X} $:
\[
\lim\inf_{\nu\to 0}\; \nu \log \PX(u^\nu \in G) \geq -I(G).
\]
\end{definition}
Let ${\mathcal C}_0=\{ \int_0^. h(s)ds \, :\, h\in L^2(0,T ;  H_0)\}
\subset {\mathcal C}([0, T]; H_0)$.
Given $\zeta\in {\mathcal Y} $ define ${\mathcal G}_\zeta^0:
{\mathcal C}([0, T]; H_0)  \to  {\mathcal X} $   by
$ {\mathcal G}_\zeta^0(g)=u_h^0 $ where  $  g=\int_0^. h(s)ds \in {\mathcal C}_0$ and  $u_h^0$
is the solution  to the
(inviscid) control equation \eqref{uh}
with initial condition $\zeta$  and $\tilde{\sigma}_0=\sigma_0$,  
and ${\mathcal G}^0_\zeta(g)=0$ otherwise. The following theorem is the main result of this section.

\begin{theorem}\label{LDP_unu}
Let  $\zeta\in {\mathcal Y}$, and  for $\nu > 0 $ let   $\s_{\nu}$
satisfy conditions {\bf (C1)}--{\bf (C3q)} for any $q\in [2,+\infty)$ and let condition {\bf (C4)}
be satisfied.
Then the solution  $(u^\nu, {\nu>0})$ to \eqref{SNS}
with initial condition $\zeta$
satisfies a large deviation principle in ${\mathcal X}$
with the good rate function
\begin{eqnarray} \label{ratefc}
I (u)= \inf_{\{h \in L^2(0, T; H_0): \; u ={\mathcal G}_\zeta^0(\int_0^. h(s)ds) \}}
\Big\{\frac12 \int_0^T |h(s)|_0^2\,  ds \Big\}.
\end{eqnarray}
\end{theorem}
In order to prove this theorem, fix  $q,p\in [ 4,\infty) $, $M>0$ and $\nu_{0}>0$,  let  $(h_\nu , 0 < \nu \leq {\nu}_{0})$
be a family of random elements
taking values in the set ${\mathcal A}_M$ defined by \eqref{AM}.
Let  $u^\nu_{h_\nu}$   be
the solution of the following corresponding stochastic controlled equation
\begin{equation}\label{unuhnu}
du^\nu_{h_\nu}(t) + \big[ \nu Au^\nu_{h_\nu} (t) + B(u^\nu_{h_\nu}(t),u^\nu_{h_\nu}(t))\big] dt
= \sqrt{\nu}\,  \s_\nu(t,u^\nu_{h_\nu}(t))\, dW(t) + {\s}_\nu (t,u^\nu_{h_\nu}(t))h_{\nu}(t)dt,
\end{equation}
with initial condition $u^\nu_{h_\nu}(0)=\xi \in {\mathcal Y}$.
Note that $u^\nu_{h_\nu}={\mathcal G}^\nu_\xi\Big(\sqrt{\nu} \big( W_. + \frac{1}{\sqrt \nu}
  \int_0^. h_\nu(s)ds\big) \Big)$
due to the  uniqueness of the solution.
The following proposition establishes the weak convergence of the family
$(u^\nu_{h_\nu})$ as $\nu\to 0$.

\begin{proposition}\label{weakconv}
Let us assume that for $\nu > 0$ the coefficients $\s_{\nu}$ 
satisfy conditions  {\bf (C1)}--{\bf (C3q)}  for all $q\in [2,+\infty)$
and that condition {\bf (C4)} holds true.
Let $\zeta $ be ${\mathcal F}_0$-measurable such that
$\EX  \big(|\zeta|_{H}^{p}+\|\zeta\|_{\mathcal Y}^{p}\big)<+\infty$
for every $p\in [2,\infty)$,
and let $h_\nu$ converge to $h$ in distribution as random elements
taking values in ${\mathcal A}_M$, where this set is defined by
\eqref{AM} and endowed with the weak topology of the space
$L_2(0,T;H_0)$. Then,  as $\nu \to 0$, the solution $u^\nu_{h_\nu}$
of \eqref{unuhnu} converges in distribution  in ${\mathcal X}$ to
the solution $u_h^0$ of \eqref{uh}.
That is, as $\nu\to 0$, the process
  ${\mathcal G}^\nu_\zeta \Big(\sqrt{\nu} \big( W_. + \frac{1}{\sqrt{\nu}} \int_0^. h_\nu(s)ds\big) \Big)$
  converges in
distribution to $ {\mathcal G}^0_\xi \big(\int_0^. h(s)ds\big)$  in ${\mathcal X}$.
\end{proposition}
\begin{proof}
\par
{\bf Step 1:} Let us decompose $u^\nu_{h_\nu} = \zeta + \sum_{i=1}^4 J_i$, where
\begin{align*}
J_1&= -\nu\int_{0}^{t}Au^\nu_{h_\nu}(s)ds, & J_2&= -\int_{0}^{t}B(u^\nu_{h_\nu}(s),u^\nu_{h_\nu}(s))ds, \\
J_3&=\sqrt{\nu}\int_{0}^{t}\s_{\nu}(s,u^\nu_{h_\nu}(s))dW(s),
& J_4&=  \int_{0}^{t} {\s}_{\nu}  (s,u^\nu_{h_\nu}(s))h_\nu(s)ds.
\end{align*}
For $\nu \in (0,\nu_0]$ we have using Minkowski's and Cauchy-Schwarz's inequalities
\begin{eqnarray*}
\|J_{1}\|^{2}_{W^{1,2}(0,T; H)}
&=&\nu \int_{0}^{T}\Big|\int_{0}^{t}Au^\nu_{h_\nu}(s)ds\Big|^{2}_{H} dt
+\nu \int_{0}^{T}|Au^\nu_{h_\nu}(t)|^{2}_{H}dt \\
&\leq& C(T,p) \nu\int_{0}^{T}|A u^\nu_{h_\nu}(s)|_H^{2}ds.
\end{eqnarray*}
Hence, using the estimate \eqref{uhnu_estimate3}, we get that  for $\nu \in (0, \nu_0]$,
\begin{equation}\label{J1W2}
{\EX   }\|J_{1}\|^{2}_{W^{1,2}(0,T; H)}\leq \tilde{C}_{1}(M,T, \nu_0)  [1+\EX \|\zeta\|^{4}] .
\end{equation}
Similarly, the upper estimate \eqref{uhnu_estimate3} implies that for all $p\in [2,\infty)$   and
$\nu \in (0, \nu_0]$,
\begin{eqnarray}\label{J1Wp}
{\EX   }\|J_{1}\|^{p}_{W^{1,p}(0,T; V')}&\leq& \nu C(T) \EX \int_0^T \|A u^\nu_{h_\nu}(s)\|_{V'}^p ds\nonumber  \\
&\leq & \nu C(T) \EX \int_0^T \| u^\nu_{h_\nu}(s)\|^p ds \leq C(T, p, \nu_0) [1+\EX \|\zeta\|^p].
\end{eqnarray}
Using again Minkowski's and H\"older's  inequalities
and the estimate \eqref{boundB}, we deduce that for $4\leq p<q<\infty$ and $\nu \in (0, \nu_0]$,
\begin{equation*}
\|J_{2}\|^{p}_{W^{1,p}(0,T; H)}\leq C(T,p,\nu_0) \int_0^T \|u^\nu_{h_\nu}(t)\|_{H^{1,q}}^{p}
\|u^\nu_{h_\nu}(t)\|^{p} dt.
\end{equation*}
Thus H\"older's inequality with the conjugate exponents $q/p$ and $q/(q-p)$ and the upper estimates
\eqref{uhnu_estimate3} and \eqref{uhnu_H1q} yield for $\nu \in (0,\nu_0]$:
\begin{equation}\label{J2Wp}
{\EX   }\|J_{2}\|_{W^{1,p}(0,T; H)}^p \leq  C(T,M,p,q) [1+{\EX   }\|\zeta\|^{pq/(q-p)}]^{1-p/q}\,
[1+{\EX}\|\zeta\|_{H^{1,q}}^{q}]^{p/q}.
\end{equation}
The Minkowski and Cauchy Schwarz inequalities and condition {\bf (C1)} imply   for $\nu \in (0, \nu_0]$
\begin{eqnarray*}  \|J_4\|_{W^{1,2}(0,T;H)}^2 &\leq& C(T) \int_0^T \| {\s}_\nu  (s,u^\nu_{h_\nu}(s))\|_{L_Q}^2 |h_\nu(s)|_0^2 ds
\\
&\leq & C(T,\nu_0,M) \Big[1+\sup_{0\leq t\leq T} |u_{h_\nu}^\nu(t)|_{H}^2
\Big].
\end{eqnarray*}
Thus the upper estimate  \eqref{uhnu_estimate1}  
yields
that for $\nu \in (0, {\nu}_0]$ one has:
\begin{equation} \label{J4W2}
\EX  \|J_4\|_{W^{1,2}(0,T;H)}^2 \leq C(T,M) \big[ 1+ \EX |\zeta|_{H}^4
   \big].
\end{equation}
Furthermore,  H\"older's inequality and  {\bf (C1)}  imply that for
$\nu \in (0, {\nu}_0]$ and $p\in [4,\infty)$:
\[ \int_0^T | J_4(t) |_H^p dt \leq M^{\frac{p}{2}} C  \big [1+
\sup_{s\leq T} |u^\nu_{h_\nu}(s)|_{H}^p
\big] .
\]

Let $\alpha\in (0,\frac{1}{2})$; then using again  Minkowski's and
H\"older's inequalities,    condition {\bf (C1)} and Fubini's theorem, we deduce that for $\nu \in (0, {\nu}_0]$:
\begin{align*}
\int_0^T &\int_0^T \frac{ |J_4(t)-J_4(s)|_H^p}{(t-s)^{1+\alpha p}} ds dt    \\
&\leq 2 \int_0^T dt \int_0^t ds (t-s)^{-1-\alpha p}
\Big| \int_s^t | {\s}_\nu  (r,u^\nu_{h_\nu}(r))|_{L_Q} |h_\nu(r)|_0 dr\Big|^p\\
&\leq C M^{\frac{p}{2} } \big [1+
\sup_{r\leq T} |u^\nu_{h_\nu}(r)|_{H}^p
\big]
\int_0^T dt \int_0^t (t-s)^{-1 + (1/2-\alpha)p} ds.
\end{align*}
The two above estimates and   \eqref{uhnu_estimate1}  
imply that for $\alpha \in (0,\frac{1}{2})$, $p\in [4,\infty)$   and $\nu \in (0, {\nu}_0]$:
\begin{equation}\label{J4Walphap}
\EX  \|J_4\|_{W^{\alpha,p}(0,T;H)}^p \leq C(p,\alpha,T,M)  \big[ 1+ \EX |\zeta|_{H}^p
\big].
\end{equation}
The Burkholder-Davis-Gundy and H\"older inequalities imply
\begin{eqnarray*}
{\EX   }\int_{0}^{T} |J_{3}(t)|^{p}_{H}dt
&\leq &C_p \nu^{p/2}\int_{0}^{T}{\EX   }\left(\int_{0}^{T}|\s_{\nu}(s,u^\nu_{h_\nu}(s))|^{2}_{L_{Q}}ds\right)^{p/2}dt \\
&\leq & C_p T^{p/2 -1} \nu^{p/2}\int_{0}^{T}{\EX   }|\s_{\nu}(s,u^\nu_{h_\nu}(t))|^{p}_{L_{Q}}dt.
\end{eqnarray*}
Let $p\in [4,\infty)$,  $\alpha \in (0, \frac{1}{2})$  and for $t\in [0,T]$ set $\phi(t):=\int_{0}^{t}|\s_{\nu}(s,u^\nu_{h_\nu}(s))|^{2}_{L_{Q}}ds$; then
the Burkholder-Davis-Gundy  and H\"older inequalities imply
\begin{align*}
{\EX   }\int_{0}^{T}\int_{0}^{T} & \frac{|J_{3}(t)-J_{3}(s)|_{H}^{p}}{|t-s|^{1+p\alpha}}dtds=
\nu^{p/2}\int_{0}^{T} dt \int_{0}^{T} ds \frac{{\EX   }|\int_{s\wedge t}^{s\vee t}\s_{\nu}(r,u^\nu_{h_\nu}(r))dW(r)|^{p}_{H}}{|t-s|^{1+p\alpha}}\\
&\leq  C_p \nu^{p/2} \int_{0}^{T}\int_{0}^{T} {\EX   }\left| \int_{s\wedge t}^{s \vee t}| \s_{\nu}(r, u^\nu_{h_\nu}(r))|^{2}_{L_Q}dr\right|^{p/2}
{|t-s|^{-(1+p\alpha)}}dtds\\
&\leq C_p \nu^{p/2}{\EX   } \int_{0}^{T}\int_{0}^{T} |\phi(t)-\phi(s)|^{p/2}|t-s|^{-(1+p\alpha)} dtds\\
&\leq C_p \nu^{p/2}{\EX   }\|\phi\|^{\frac{p}{2}}_{W^{2\alpha,p/2}(0,T; \mathbb{R})}  \\
& \leq C_p \nu^{p/2}{\EX   }\|\phi\|^{\frac{p}{2}}_{W^{1,p/2}(0,T; \mathbb{R})}\\
&\leq C_p C(T)  \nu^{p/2} \EX \int_{0}^{T}|\s_{\nu}(s, u^\nu_{h_\nu}(s))|^{p}_{L_{Q}}ds.
\end{align*}
Using the assumption {\bf (C1)} and the two above upper estimates of $J_3$, we deduce that
\[
\EX \|J_3\|_{W^{\alpha,p}(0,T;H)}^p \leq C(p,T) \nu^{p/2}
\big[ 1+ \sup_{0\leq t\leq T} \EX |u^\nu_{h_\nu}(t)|_{H}^p
\big].
\]
Finally, the upper estimate \eqref{uhnu_estimate1}
yields
for $\nu \in (0,{\nu}_0]$ and $p\in [4,\infty)$:
\begin{equation} \label{J3Walphap}
\EX \|J_3\|_{W^{\alpha,p}(0,T;H)}^p \leq C(p,T) {\nu}_0^{p/2}  \big[ 1+ \EX |\zeta|_{H}^p
\big].
\end{equation}
Collecting all the estimates \eqref{J1W2}-\eqref{J3Walphap}  we deduce  that for $p\in [4,\infty)$,
$\alpha\in (0,1/2)$, there exists a positive constant $C(p,M,T)$ such that for any  $\nu\in (0,{\nu}_{0}]$
\begin{equation}\label{estimunuhnu}
\EX \|u^\nu_{h_\nu}\|_{W^{\alpha,2}(0,T; H)}^2 +  \EX \|u^\nu_{h_\nu}\|_{W^{\alpha,p}(0,T; V')}^p \leq C(p,M,T).
\end{equation}

\par
{\bf Step 2:} The upper estimates
\eqref{uhnu_estimate3} and    \eqref{estimunuhnu}  show  that the process
$(u^\nu_{h_\nu}, \nu \in (0, {\nu}_0])$ is bounded in probability in
\[
W^{\alpha,2}(0,T;H ) \bigcap L^{2}(0,T; V) \bigcap W^{\alpha,p}(0,T; V').
\]
Thanks to the compactness theorem given in \cite{Lions70}, Chapter 1, Section 5. the space
$W^{\alpha,2}(0,T; H)\bigcap L^{2}(0,T; V)$ is compactly embedded in $L^{2}(0,T; \mathcal{H})$.
For $p\alpha>1$, thanks to  Theorem 2.2  given in \cite{FG} (see also  \cite{BP2001} and the references therein),
the space
$W^{\alpha,p}(0,T; V')$ is compactly embedded in ${\mathcal C}([0,T]; D(A^{-\beta}))$ with $2\beta>1$.

On the other hand, the family $(h_\nu)$ is included in
${\mathcal A}_M$. Set $F_\nu(t)=\int_0^t h_\nu(s) ds$;
since $H_0$ is compactly embedded in $H$, we can again use the above compact embedding theorem and deduce
that $W^{1,2}(0,T;H_0 )$ is compactly embedded in ${\mathcal C} ([0,T]; H)$.
Furthermore, by assumption
$h_\nu\to h$  in distribution   in $L^2(0,T ;H_{0})$ endowed with the weak topology. This yields
that $F_\nu \to F$ in distribution in the weak topology of $W^{1,2}(0,T;H_0)$,
denoted by $W^{1,2}(0,T;H_0)_w$, where $F(t):=\int_0^t h(s) ds$.

Hence, by the Prokhorov theorem,
the family of distributions  $(\mathcal{L}(h_\nu, u^\nu_{h_\nu}, \nu \in (0, {\nu}_0])$
of the process $(F_\nu, u^\nu_{h_\nu}, \nu\in (0, {\nu}_0])$
is tight in

\[ {\mathcal Z}: = \Big[ W^{1,2}(0,T;H_0)_w\bigcap {\mathcal C}([0,T], H)\Big]
\times \Big[ L^{2}(0,T;
\mathcal{H})\bigcap {\mathcal C}([0,T]; D(A^{-\beta}))\Big] .\]
Let $(\nu_n , n\geq 0)$ be a sequence in $(0,{\nu}_0]$ such that $\nu_n \to 0$.
Thus, we can extract a subsequence, still denoted by $(F_{\nu_n}, u^{\nu_n}_{h_{\nu_n}})$, that
converges in distribution in  ${\mathcal Z}  $ to a pair $(\bar{F},\bar{u})$ as $n \to \infty$.
Note that by assumption, $\bar{F}=F$.
\medskip

\par
{\bf Step 3:}  By the Skorohod-Jakubowski Theorem,  \cite{Jaku97} Theorem 2, recalled in the Appendix (see also \cite{BrzOnd2011}), there exists a stochastic basis
$(\Omega^{1}, \mathcal{F}^{1}, (\mathcal{F}^{1}_{t}), \PP^{1})$ and on
this basis,   $ {\mathcal Z} $-
valued random variables $(F^1=\int_0^. h^1(s) ds, u^1)$ and for
$n\geq 0$  $(F^{\nu_n,1}=\int_0^. h^{\nu_n,1}(s) ds, u^{\nu_n,1}_{h^{\nu_n,1}})$,
such that  the pairs  $(F^1,u^1)$ and $(\bar{F},\bar{u})$ have the same distribution, for $n\geq 0$
the pairs
$(F^{\nu_n,1}, u^{\nu_n,1}_{h^{\nu_n,1}}) $  and $(F_{\nu_n}, u^{\nu_n}_{h_{\nu_n}})$ have  the
same distribution on ${\mathcal Z}$, and as $n \to \infty$,
$(F^{\nu_n,1}, u^{\nu_n,1}_{h^{\nu_n,1}}) \longrightarrow (F^1,u^{1})$ in
${\mathcal Z}$ ${\mathbb P}^1$ a.s To ease notations in the sequel, we will skip the upper index 1 and the index $n$ of the
subsequence  and still denote $F^{1,\nu}$ by $F_\nu$, $h^{1,\nu}$ by
$h_\nu$, $u^{1,\nu}_{h^{1,\nu}}$ by $u^\nu_{h_\nu}$, $F^1$ by $F$, $h^1$ by $h$ and  $u^1$ by $\bar{u}$.
Let again $\zeta$ denote the initial condition $u^{\nu,1}_{h^{1,\nu}}(0)$. \\
Moreover, by \eqref{uhnu_estimate1}, \eqref{uhnu_estimate3} and \eqref{uhnu_H1q} we deduce the existence of constants $C_i$
such that for $\nu \in (0,\nu_0]$, 
$\alpha \in (0,1/2)$
and   $q\in [2,\infty)$: 

\begin{align*} &  \EX_1 \Big( \sup_{0\leq t\leq T}|u^{\nu}_{h_\nu}(t)|^{2}_{H}\Big) \leq C_{1},\quad
\EX_1 \int_{0}^{T}\|u^{\nu}_{h_\nu}(t)\|^{2} dt \leq C_{2},\quad
\EX_1 \Big( \sup_{0\leq t\leq T}\|u^{\nu}_{h_\nu}(t)\|^{q}_{H^{1,q}(D)}\Big) \leq C_{3} .
\end{align*} 
Therefore,  we can extract a further subsequence which converges weakly to $\bar{u}$ in 
$L^{2}(\Omega^1\times (0,T); V) \bigcap L^q(\Omega^1\times (0,T); H^{1,q})$ as
$n \to \infty$.
This implies that
\begin{equation}\label{u1_space}
\bar{u} \in L^{2}(0,T; V)\bigcap  L^{\infty}\big(0,T; H\cap
H^{1,q}\big)\ \ \PP^{1}-{\rm a.s.}
\end{equation}

\par
{\bf Step 4:} (Identification of the limit.) We have to prove that
the limit $\bar{u}$ is solution of the equation
\begin{equation}\label{uh1}
d\bar{u}(t) + B(\bar{u}(t),\bar{u}(t))\, dt = {\sigma}_{0}(t, \bar{u}(t))\,
h(t)\, dt\; , \quad \bar{u}(0)= {\zeta} .
\end{equation}
Let $\varphi\in D(A^{\beta})$ with $2\beta>1$; then
\begin{equation} \label{identif}
  (u^\nu_{h_\nu}(t)-\zeta ,\varphi)
  + \int_0^t\!\!\!  \big\langle B(\bar{u}(s),\bar{u}(s)) - \s_{0}(s,\bar{u}(s)) h(s)  ,
\varphi\big\rangle  ds  = \sum_{1\leq i\leq 6} I_i,
\end{equation}
where
\begin{align*} I_1&=-\nu\int_{0}^{t}\left(
Au^\nu_{h_\nu}(s),\varphi\right) ds , \quad I_2=
\sqrt{\nu}\int_{0}^{t}\left(\s_{\nu}(s,u^\nu_{h_\nu}(s))dW(s),\varphi\right) , \\
I_3& =  - \int_0^t\!\!\! \big[\left\langle
B(u^{\nu}_{h_\nu}(s)-\bar{u}(s),u^{\nu}_{h_\nu}(s)),\varphi\right\rangle +
\left\langle
B(\bar{u}(s) , u^{\nu}_{h_\nu}(s)-\bar{u}(s)),\varphi\right\rangle  \big]  ds, \\
I_4 &= \int_0^t \left(
\big[ \s_\nu(s,u^{\nu}_{h_\nu}(s)) - \s_0(s,u^{\nu}_{h_\nu}(s))\big]  h_{\nu}(s),\varphi\right) ds,  \\
I_5 & = \int_0^t \left( \big[
\s_{0}(s,{u}^\nu_{h_\nu}(s)) -\s_{0}(s,\bar{u}(s))\big] h_{\nu}(s),\varphi\right) ds ,  \\
I_6 &= \int_0^t \left(
\s_{0}(s,\bar{u}(s))\left[h_{\nu}(s)-h(s)\right],\varphi\right) ds.
\end{align*}
Since $\beta \geq 1/2$    implies that $Dom (A^\beta)\subset V$, using  Cauchy-Schwarz's inequality and  \eqref{uhnu_estimate3},
we deduce for $t\in [0,T]$ and $\nu \in (0, {\nu}_0]$:
\begin{align}\label{errorI_{1}}
{\EX   }_1|I_{1}| &\leq \nu {\EX}_1\int_{0}^{t}\|u^{\nu}_{h_\nu}(s)\|\|\varphi\|ds
\leq  \nu\sqrt{t}\|\varphi\|\left( {\EX   }_1\int_{0}^{t}\|u^{\nu}_{h_\nu}(s)\|^{2}ds\right)^{1/2} \nonumber \\
&\leq  {\nu}  C(T,M) \|\varphi\|
\big[ 1+\EX \|\zeta\|^4\big]^{1/2} .
\end{align}
The  It\^o isometry, the Cauchy-Schwarz's inequality,  condition {\bf (C1)} and \eqref{uhnu_estimate1} yield
\begin{align}\label{errorI_{2}}
{\EX   }_1 |I_{2}|&
\leq \sqrt{\nu} {\EX   }_1 \left(\int_{0}^{t}|\s_{\nu}(s, u^{\nu}_{h_\nu}(s))|^{2}_{L_{Q}}\|\varphi\|^{2}\right)^{1/2}\nonumber\\
&\leq  \sqrt{\nu} \|\varphi\| C(T,M) \big[1+\EX |\zeta|_{H}^4\big]^{1/2}.
\end{align}
Using \eqref{B(u,u)L4},  the Cauchy-Schwarz inequality and \eqref{uhnu_estimate3} we get
\begin{align}\label{errorI_{3}}
{\EX   }_1|I_{3}|&\leq
C{\EX   }\int_{0}^{t}\|u^{\nu}_{h_\nu}(s)-\bar{u}(s)\|_{\mathcal{H}}\left(\|u^{\nu}_{h_\nu}(s)\|_{\mathcal{H}}
+\|\bar{u}(s)\|_{\mathcal{H}}\right)\|\varphi\|ds\nonumber\\
&\leq C\|\varphi\|\left({\EX   }_1\int_{0}^{t} \!\!\! \|u^{\nu}_{h_\nu}(s)-\bar{u}(s)\|^{2}_{\mathcal{H}} ds \right)^{1/2}
\left({\EX   }_1 \int_{0}^{t}\!\!\! \big[ \|u^{\nu}_{h_\nu}(s)\|^{2}_{V}
+\|\bar{u}(s)\|^{2}_{V}\big] ds \right)^{1/2} \nonumber \\
&\leq C(T,M) \|\varphi\| \big[1+\EX\|\zeta\|^4  \big]^{1/2} \left({\EX   }_1 \int_{0}^{t}\! \|u^{\nu}_{h_\nu}(s) -
\bar{u}(s)\|^{2}_{\mathcal H} ds \right)^{1/2} .
\end{align}
Using   assumption {\bf (C4)},  the Cauchy Schwarz inequality and 
  \eqref{uhnu_estimate1}  we obtain  
\begin{align}\label{errorI_{5}}
{\EX   }_1 |I_{4}|&\leq 
{\EX }_1\int_{0}^{t}\big| {\s}_{\nu}(s,u^{\nu}_{h_\nu}(s)) - \s_0(s,u^{\nu}_{h_\nu}(s)) \big|_{L(H_0,H)} |h_\nu(s)|_0
\, |\varphi|_{H}ds\nonumber\\
&\leq   
|\varphi|_{H} \; \sqrt{MT} \;
\Big(  {\EX   }_1 \int_{0}^{t}\big|  {\s}_{\nu}(s, u^{\nu}_{h_\nu}(s)) - \s_0(,u^{\nu}_{h_\nu}(s))  \big|^2_{L(H_0,H)}
ds\Big)^{1/2} \nonumber\\
&\leq 
C(\nu)\,  |\varphi|_{H}\, \sqrt{MT}\,
\left( \EX_1    \int_{0}^{t}\big[ 
1+ |u^\nu_{h_\nu}( s)|_H^{2}\big]ds \, \right)^{1/2}\nonumber\\
&\leq 
C(\nu)  \, |\varphi|_{H}\, C(T,M) \,   \big[1+ \EX_1 |\zeta| ^2_{H}  \big]^{1/2}                .
\end{align}

Condition {\bf (C1)} and  the Cauchy Schwarz inequality yield the existence of $\bar{L}_1>0$
such that for $\nu \in (0, \nu_0]$
\begin{align}\label{errorI_{6}}
{\EX }_1|I_{5}|&\leq {\EX}_1 \int_{0}^{t}\left| {\s}_{0}(s,u^\nu_{h_\nu}) - \s_{0}(s,\bar{u}(s))\right|_{L(H_0,H)} |h_{\nu}(s)|_0
|\varphi|_{H}  ds\nonumber\\
&\leq  |\varphi|_{H}\, \sqrt{MT} \left( \EX_1 \int_{0}^{t} |\s_{0}(s,u^{\nu}_{h_\nu}(s)) - \s_{0}(s,\bar{u}(s))|_{L_{Q}}^2 ds\right)^{1/2} \nonumber\\
&\leq  |\varphi|_{H} \sqrt{MT} \sqrt{\bar{L}_1} \left( \EX_1   \int_{0}^{t}|u^\nu_{h_\nu}(s) - \bar{u}(s)|_{H}^{2}ds\right)^{1/2} .
\end{align}
Finally, we have that
\begin{equation}\label{errorI_{7}}
{\EX   }_1|I_{6}|= {\EX}_1\left|\int_0^t \Big(
\left[h_{\nu}(s)-h(s)\right],\s_{0}^{*}(s,\bar{u}(s))\varphi\Big) ds \right|.
\end{equation}
Using the upper estimates \eqref{errorI_{1}}, \eqref{errorI_{2}},  \eqref{errorI_{5}}
and \eqref{errorI_{6}}  we deduce that $\EX_1 |I_i|\to 0$
for $i=1,2,4$ as $n \to \infty$ and $\nu_n \to 0$.   Furthermore, by construction, we have  $\PP^1$ a.s.
$u^{\nu_n}_{h_{\nu_n}} - \bar{u} \to 0$ in $L^2(0,T;H^{1,4})$ and hence in
$L^2(0,T;{\mathcal H})$ and in $L^2(0,T;H)$  as $n \to 0$.
Furthermore, the estimates \eqref{uhnu_estimate3},  \eqref{uh_estimate1}  prove that
$\int_0^T \|u^\nu_{h_\nu}(s) - \bar{u}(s)\|^2 ds$ is bounded in $L^2(\PP^1)$  and hence is uniformly integrable. Therefore, the dominated
convergence theorem and  \eqref{errorI_{3}}  prove that $\EX_1|I_i|\to 0$ for $i=3,5$.
Finally, condition {\bf (C1)} shows that
\[ \int_0^T |\s_{0}^*(s,\bar{u}(s)) \varphi |_0^2 ds \leq
|\varphi|_{H}^2 \int_0^T \big[ \bar{K}_0 + \bar{K}_1  |\bar{u}(s)|_{H}^2\big] ds
\]
and  by assumption, as $n \to \infty$,  
we have  $h_{\nu_n} -h\to 0$ in $L^2(0,T;H_0)$ for the weak topology $\PP^1$ a.s.
Hence $\PP^1$ a.s., $\int_0^t \left(\left[h_{\nu_n}(s)-h(s)\right],\s_{0}^{*}(s,\bar{u}(s))\varphi\right) ds$ converges
to 0 as $n \to \infty$. Furthermore, the upper estimate  \eqref{uh_estimate1} proves that this family is bounded in $L^2(\PP^1)$; using
once more the dominated convergence theorem,
\eqref{errorI_{7}} proves that $\EX_1 |I_6|\to 0$ as $n \to \infty$.
Thus, \eqref{identif} shows that as $n \to \infty $,   for any $t\in [0,T]$ and $\varphi \in D(A^\beta)$
with $\beta >1/2$:
\begin{equation}  \label{ident1}
\EX_1 \Big[  (u^{\nu_n}_{h_{\nu_n}}(t) ,\varphi) - \int_0^t \big\langle  -B(\bar{u}(s),\bar{u}(s)) + \s_{0}(s,\bar{u}(s)) h(s)  ,
\varphi\big\rangle  ds  \Big]  \to 0 .
\end{equation}
On the other hand, by construction, since $\varphi \in Dom(A^\beta)$, we have $\PP^1$ a.s.
\[ \sup_{t\in [0,T]}
| ( u^{\nu_n}_{h_{\nu_n}}(t) - \bar{u}(t) , \varphi ) |\to 0 \quad  \PP^1 \;   \mbox{\rm a.s. as } \nu \to 0.\]
Using again  \eqref{uhnu_estimate3},  \eqref{uh_estimate1} and the dominated convergence theorem, we deduce that as $n \to \infty$,
\begin{equation} \label{ident2}
\EX_1 \Big( \sup_{t\in [0,T]}
| (u^{\nu_n}_{h_{\nu_n}}(t)  - \bar{u}(t) , \varphi) |\Big) \to 0.
\end{equation}
Since $\PP^1$ a.s. $\bar{u}\in \mathcal{C}([0,T], D(A^{-\beta}))$,
this identity holds a.s. for all $t\in [0,T]$  and $\bar{u}$ is a solution to the inviscid
evolution equation  \eqref{uh}.
Thus the uniqueness of the solution to \eqref{uh}
proved in Theorem \ref{uniquness_uh} implies that $\bar{u}=u_h^0$.
Theorems \ref{existence_uh} and \ref{uniquness_uh} prove that $u^0_h$ belongs to $C([0,T];H)\cap L^\infty(0,T; V\cap H^{1,q})$.
Hence, from any sequence $\nu_n \to 0$, one can extract a subsequence $(\nu_{n_k} , k\geq 0)$ such that
$u^{\nu_{n_k}}_{h_{\nu_{n_k}}} \to u^0_h$ in distribution in ${\mathcal X}$. This implies that the family
$u^\nu_{h_\nu}$ converges to $u^0_h$ in distribution in ${\mathcal H}$, which concludes the proof.
\end{proof}

The following compactness result is the second ingredient which
allows to transfer the  LDP  from $\sqrt{\nu} W$ to $u^\nu$.
\begin{proposition}\label{compact}
Suppose that $\tilde{\s}_{0}$ satisfies condition {\bf (C1Bis)}, {\bf (C2Bis)} and
{\bf (C3qBis)}  for all $q\in [2,+\infty)$.
Fix  $M>0$, $\zeta\in
{\mathcal Y}$ and let $ K_M=\{u_h^0  :  h \in S_M \}$, where $u_h^0$
is the unique solution in ${\mathcal X}$ of the deterministic
control equation \eqref{uh}. Then $K_M$ is a compact subset of
${\mathcal X}$.
\end{proposition}
\begin{proof}
To simplify the notation, we skip the superscript 0 which refers to the inviscid case.
By Theorems \ref{existence_uh} and  \ref{uniquness_uh},  $K_M\subset {\mathcal X}$.
Let $(u_n, n\geq 1)$ be a sequence in $K_M$, corresponding to solutions of
(\ref{uh}) with controls $(h_n, n\geq 1)$ in $S_M$:
\begin{eqnarray*}
d u_n(t) + B(u_n(t), u_n(t))dt =\tilde{\s}_{0}  (t,u_n(t)) h_n(t) dt, \;\;
u_n(0)=\zeta.
\end{eqnarray*}
Since $S_M$ is a bounded closed subset of the Hilbert space $L^2(0,
T; H_0)$, it is weakly compact. So there exists a subsequence of
$(h_n)$, still denoted as $(h_n)$, which converges weakly to a
limit $h \in L^2(0, T; H_0)$. Note that in fact $h \in S_M$ as $S_M$
is closed.

We at first prove that $(u_n)$ is bounded in
$W^{1,2}(0,T;L^q) \cap W^{\alpha,p}(0,T; L^q) \cap L^2(0,T; H^{1,q})$ for any $p,q>2$ and $\alpha < \frac{1}{2}$.
Indeed, $u_n(t)= \zeta + J_1(t) + J_2(t)$, where
\[ J_1(t)= -\int_0^t B(u_n(s),u_n(s)) ds,\quad J_2(t)=\int_0^t \tilde{\s}_{0}(s,u_n(s)) h_n(s) ds.\]
H\"older's  inequality,   \eqref{boundB},
and \eqref{uh_gradiantq}  yield
\begin{equation} \label{I1inviscid}
\|J_1\|_{W^{1,q}(0,T, L^q)}^q    \leq C(T) \sup_{t\in [0,T]} \|u_n(t)\|_{H^{1,q}}^{2q}
\leq  q^{2q}  C(T,M) [1+\|\zeta\| + \|\mbox{\rm curl }\zeta\|_q]^{2q} .
\end{equation}
Furthermore, Minkowski's inequality, the Sobolev embedding theorem (see \eqref{Sobolevq}),
\eqref{normcurl}, condition {\bf (C2Bis)} and \eqref{uh_estimate1}  yield
\begin{align} \label{I2invisci-2}
\|J_2&\|_{W^{1,2}(0,T;L^q)}^2 \leq C(T,q) \int_0^T\!\!\!  \|  \tilde{\s}_{0}(t,u_n(t)) h_n(t)\|_q^2 dt \leq C(T,q)
\int_0^T \!\!\!  \|   \tilde{\s}_{0}(t,u_n(t)) h_n(t)\|^2  dt\nonumber \\
& \leq C(T,q) \int_0^T \!\!  | \mbox{\rm curl } \tilde{\s}_{0}(t,u_n(t))|_{L(H_0,H)}^2
| h_n(t)|_0^2 dt \leq C(T,q) M \big[ \tilde{K}_0  + \tilde{K}_1 \sup_{t\in [0,T]} \|u(t) \|^2] \nonumber \\
&\leq C(T,q,M) [1+\|\zeta\|^2].
\end{align}
The Minkowski and H\"older inequalities, the Sobolev embedding theorem,  \eqref{normcurl},
conditions {\bf (C1Bis)}, {\bf(C2Bis)}, \eqref{uh_estimate1} and \eqref{Sobolevq} imply
\begin{align} \label{I2inviscid-q}
\int_0^T&\!\! \|J_2(t)\|_q^p dt \leq \int_0^T \!\! \Big| \int_0^t \!\!
\| \tilde{\s}_{0}(s, u_n(s)) h_n(s) \|_q ds \Big|^p dt
\leq  C(q)   \int_0^T \!\! \Big| \int_0^t \!\! \| \tilde{\s}_{0} (s, u_n(s)) h_n(s) \| ds \Big|^p dt \nonumber \\
&\leq  C(q)  (MT)^{p/2}  \sup_{t\in [0,T]} | \tilde{\s}_{0}(t,u_n(t))|_{L(H_0,V)}^p
\leq  C(p,q,T,M)
\big[ 1   +  \sup_{t\in [0,T]} \|u(t) \|^p \big]   \nonumber \\
&\leq C(p,q,T,M) [1+\|\zeta\|^p].
\end{align}
Finally, similar arguments imply that for $\alpha \in (0, \frac{1}{2})$, we have
\begin{align}\label{I2inviscid-Wq}
\int_0^T &\int_0^T \frac{ \|J_2(t)-J_2(s)\|_q^p}{(t-s)^{1+\alpha p}} ds dt   \nonumber \\
&\leq 2 C(q)  \int_0^T dt \int_0^t ds (t-s)^{-1-\alpha p}
\Big| \int_s^t \| \tilde{\s}_{0} (r,u_n(r))h_n(r)\| ds\Big|^p\nonumber \\
&\leq  2  C(q) (T M)^{\frac{p}{2} } \, C\,   \big [1+
\sup_{ r\in [0, T]} \|u_n(r)\|^p \big]
\int_0^T dt \int_0^t (t-s)^{-1 + (1/2-\alpha)p} ds \nonumber \\
&\leq C(q,T,M)   \big[1+\|\zeta\|^p\big].
\end{align}
As in the proof of Proposition \ref{weakconv}, Step 3, using \cite{Lions70} we deduce from the upper estimates
\eqref{I1inviscid}-\eqref{I2inviscid-Wq}  that the sequence $(u_n)$
is relatively compact in $ L^2(0,T ; {\mathcal H}) \cap {\mathcal C}([0,T],D(A^{-\beta}) )$ with $2\beta >1$.
Hence there exists a subsequence, still denoted $(u_n)$,
which converges in $ L^2(0,T ; {\mathcal H}) \cap {\mathcal C}([0,T],D(A^{-\beta}) )$ to some element $u$.
It remains to check that $u$ is the solution to the evolution equation
\[ du(t)+ B(u(t),u(t)) dt = \tilde{\s}_{0}(t,u(t)) h(t) dt, \quad  u(0)=\zeta.\]
The proof, which is similar to that of Step 4 in  Proposition \ref{weakconv} and easier, is briefly sketched.
Only (deterministic) terms similar to $I_i$ for $i=3,5$ and $6$ have to be dealt with.
As in the proof of Proposition \ref{weakconv}, these terms are estimated
replacing the upper estimate \eqref{uhnu_estimate3} by \eqref{uh_estimate1}.
This concludes the proof of the Proposition.
\end{proof}

The proof of Theorem \ref{LDP_unu} is a straightforward consequence of Propositions \ref{weakconv} and \ref{compact},
as shown in \cite{BD07}.

\section{Appendix} \label{Appendix}
\subsection{Properties of the bilinear operator}\label{AppB}
Let us at first recall the following classical Sobolev embeddings
which hold since $D$ is a bounded domain of $\RR^2$ which
satisfies the cone condition (see e.g. \cite{adams}):
\begin{align} \label{Sobolevq}
\|u\|_q \leq C(q) \|u\|_{W^{1,2}} \quad \mbox{\rm for } u\in W^{1,2}\quad \mbox{\rm and }  1\leq  q<+\infty, \\
W^{2,1} \subset {\mathcal C}^0_B(D), \;\,  W^{1,q}\subset
{\mathcal C}^0_B(D)  \; \mbox{\rm for}\; q\in (2,\infty).\label{Sobolev}
\end{align}
Furthermore, recall the following result proved in \cite{KaPo}
(see also \cite{BP2001} and \cite{Yu} for the way the constant depends on $q$).
Given  $q\in [2,\infty)$   there exists a
constant $C$ such that for every  $u\in H^{1,q}$ one has:
\begin{equation}\label{normcurl}
\|\nabla u\|_q \leq C q  \|\mbox{\rm curl } u\|_q \;  \mbox{\rm for } q\in [2,\infty) .
\end{equation}
Furthermore, given $q\in [2,\infty)$ and $r>0$, the operator $B$
has
a unique  extension to a continuous bilinear operator
from $H^{1,q}\times H^{1,q}$ to $H^{-r,q}$ and the following
estimates are satisfied for some constant $C$ and all $u,v\in
H^{1,q}$  resp. $\varphi, \psi \in D(A)$:
\begin{align}
& \|B(u,v)\|_{H^{-r,q}} \leq C \, \|u\|_{H^{1,q}}\,   \|v\|_{H^{1,q}},  \label{normBrq}\\
& \langle B(u,v)\, ,\, v\rangle =0
,\label{Bnul} \\
&\langle \mbox{\rm curl }  B(\varphi,\varphi),\psi\rangle
=\langle \varphi \cdot \nabla(\mbox{\rm curl }  \varphi),\psi\rangle = \langle B(\varphi,\mbox{\rm curl } \varphi ),\psi\rangle  ,
\label{curlB} \\
& \langle \mbox{\rm curl } B(u,v) \, ,\,  \mbox{\rm curl }v |
\mbox{\rm curl }v |^{q-2}\rangle  =0 \quad \mbox{\rm for all }
u,v\in  H^{2,q}\bigcap D(A).\label{curlBnul}
\end{align}
Finally, if $q>2$, there exists a constant $C>0$ such that for all
$u,v\in H^{1,q}$
\begin{align}
& |B(u,v)|_H  \leq C \, \|u\|_{H^{1,q}} \,   \|v\|_{H^{1,2}} \quad
\mbox{\rm and }\;  \|B(u,v)\|_q  \leq C \, \|u\|_{H^{1,q}} \,
\|v\|_{H^{1,q}} .
\label{boundB}
\end{align}

\subsection{Radonifying operators and stochastic calculus in $W^{k,q}$ spaces}\label{Radon}
In this section, we recall the basic definitions and results of
stochastic calculus on non Hilbert Sobolev spaces used in this
paper. Their proofs can be found in references \cite{BP99},
\cite{BP2001}, \cite{Det},  \cite{N78} and \cite{O}.

Let ${\mathcal E}$ be a Banach space, such as the Sobolev spaces  $W^{k,q}$
for $k\geq 0$ and $q\in [1,\infty)$,  and let  $H_0$ be a
Hilbert space. The following notion extends that of Hilbert
Schmidt operator from $H_0$ to ${\mathcal E}$ when ${\mathcal E}$
is not a Hilbert space. Let $(e_k)$ denote an orthonormal basis of
$H_0$ and $(\beta_k)$ be a sequence of independent standard
Gaussian random variables on some probability space $(\tilde{\Omega}, \tilde{\mathcal F}, \tilde{P} )$.
\begin{Def}
A linear operator $K:H_0\to {\mathcal E}$ is Radonifying if the
series $\sum_k \beta_k Ke_k$ converges in $L^2(\tilde{\Omega}  \; {\mathcal
E})$. Let $R(H_0,{\mathcal E})$ denote the set of Radonifying
operators, and given  $K\in R(H_0,{\mathcal E})$, set
\begin{equation} \label{normRadon} \|K\|_{R(H_0,{\mathcal E})}= \Big(\tilde{\EE}\Big|\sum_k \beta_k
Ke_k\Big|_{\mathcal E}^2\Big)^{\frac{1}{2}}.
\end{equation}
\end{Def}
Then  $(R(H_0,{\mathcal E}), \|K\|_{R(H_0,{\mathcal E})})$ is a
separable Banach space and $\|K\|_{R(H_0,{\mathcal E})} $ does not depend on the choice
of $(e_k)$ and $(\beta_k)$.

We now suppose that $H_0$ is the RKHS of the $H$-valued Wiener
process $(W(t), t\geq 0)$ and fix some orthonormal  basis $(e_k)$ of $H_0$.
Simple $R(H_0, {\mathcal E})$- valued  processes $\sigma$ on
$[0,T]$ are defined as follows.  Given integers $m,n \geq 1$,
$0\leq t_1<t_2<\cdots<t_{m+1}\leq T$, and  $\big( \sigma_j \in
L^2(\Omega,{\mathcal F}_{t_j} ; R(H_0,{\mathcal E})), j=0, \cdots,
m \big)$ set
\[\sigma(t,\omega):=\sum_{0\leq j \leq m} \sigma_j(\omega)
1_{(t_j, t_{j+1}]}(t).\]
For such a simple process $\sigma$, and
$t\in (0,T)$, set
\[
\int_0^t \sigma(s) dW_s :=\sum_{0\leq j\leq m} \sigma_j(\omega) Q^{\frac{1}{2}}
\big( W(t_{j+1}\wedge t)- W(t_j\wedge t)\big).\]
The extension of stochastic integrals to
predictable square integrable processes cannot be done for any
Banach space ${\mathcal E}$. Fix $k\in [0,\infty)$ and $q\in
[2,\infty)$ and let ${\mathcal E}=W^{k,q}$
(with the convention $L^q=W^{0,q}$).  The stochastic integral can be extended uniquely
as a linear bounded operator from the set of predictable processes
in $L^2(0,T;R(H_0, H^{k,q}))$ to the set of $({\mathcal F}_t)$
adapted random variables in $L^2(\Omega, H^{k,q}) $. Moreover, the
following Burkholder-Davies-Gundy inequality holds (see e.g.
\cite{O}, section 5): For any $p\in [1,\infty)$, there exists a constant
$C_p>0$ such that for any predictable process  $\sigma\in
L^2(0,T;R(H_0,H^{k,q}))$,
  \begin{equation}\label{BDGRadon}
\EE\Big(\sup_{0\leq t\leq T} \Big|\int_0^t  \sigma(s)
dW_s\Big|_{H^{k,q}}^p \Big) \leq C_p\, \EE\Big( \int_0^T
\|\sigma(s)\|_{R(H_0,H^{k,q})}^2\, ds\Big)^{\frac{p}{2}}
\end{equation}
Finally, given $2\leq q\leq p< \infty$,  some predictable processes $\sigma\in
L^2(0,T; R(H_0,H^{0,q}))$ and $f\in L^1(0,T; H^{0,q})   $,
we state a particular case
of the It\^o formula applied to the function
$\Psi_{q,p}(.)=\|.\|^p_q$ on $H^{0,q}$ and the $H^{0,q}$-valued process
$(Z_t, t\in [0,T])$ defined by
\[Z(t)=Z(0)+\int_0^t \sigma(s) dW(s)
+ \int_0^t f(s) ds . \]
With the above notations,  if $\langle F,G\rangle$ denotes the duality between $F\in L^q$ and $G\in L^{q*}$
with $q*=\frac{q}{q-1}$, we have:
\begin{align}\label{ItoRadon}
& \|Z(t)\|_q^p =  \|Z(0)\|_q^p + p \int_0^t \|Z(s)\|_q^{p-q} \langle |Z(s)|^{q-2} Z(s)\, ,\, f(s)\rangle ds
\nonumber \\
&\quad + p \int_0^t \|Z(s)\|_q^{p-q} \langle |Z(s)|^{q-2} Z(s)\, ,\, \sigma(s) dW(s) \rangle
+ \frac{1}{2} \int_0^t \mbox{\rm tr}_{\sigma(s)} \Psi_{q,p}''(Z(s)) ds,
\end{align}
and for every $u\in H^{0,q}$,
\begin{equation} \label{Itocor}
0\leq \mbox{\rm tr}_{\sigma(s)} \Psi_{q,p}''(u) \leq p(p-1) \, \|u\|_q^{p-2}\, \|\sigma(s)\|_{R(H_0,H^{0,q})}^2.
\end{equation}
\subsection{Nemytski operators}\label{Nemitski_op} In this section we will show that assumptions
{\bf (C1)} -- {\bf (C3qBis)}  are satisfied by Nemytski operators.
\begin{definition}
Let $q\in [2,\infty)$. A mapping $g: [0,T]\times D\times
\mathbb{R}^{2}\longrightarrow \mathbb{R}^{2}$ belongs to the class
$U(D, q)$ if and only if $g(t,x,y)=g^{1}(t,x)+g^{2}(t,x,y)$, $t\in
[0,T]$, $x\in D$, $y\in \mathbb{R}^{2}$, where:
\begin{enumerate}
\item $g^{1}$ and $g^{2}$ are measurable, and for any $t\in [0,T]$,
$g^{1}(t,\cdot)\in H^{1,2}\cap H^{1,q}$   and $g^{2}(t,\cdot,\cdot)$
is differentiable,
\item there are a constant $c>0$ and $\phi\in L^{2}(D)\cap L^{q}(D)$
such that all $t\in [0,T]$, and $x\in D$, $y\in \mathbb{R}^{2}$,
$$|g^{1}(t,\cdot)|_{H^{1,2}}+|g^{1}(t,\cdot)|_{H^{1,q}}\leq c,$$
$$|g^{2}(t,x,y)|+\sum_{i=1,2} |\partial _{x_{i}}g^{2}(t,x,y)|\leq
c(\phi(x)+|y|),\quad \sum_{i=1,2}|\partial _{y_{i}}g^{2}(t,x,y)|\leq c.$$
\end{enumerate}
We say that $g: [0,T]\times D\times \mathbb{R}^{2}\longrightarrow
\mathbb{R}^{2}$ belongs to the class $U(D, \infty)$ if and only if
it is differentiable with respect to the second and third variables,
and there is a constant $c>0$ such that for all $t\in [0,T],\ x\in D,\ y\in \mathbb{R}^{2}$:
$$|g(t,x,y)|+\sum_{i=1,2}|\partial _{x_{i}}g(t,x,y)|+ \sum_{i=1,2}|\partial _{y_{i}}g(t,x,y)|\leq
c.$$
\end{definition}
Let $g_{i}$, $i=1, \cdots, m$  and $\tilde{g}$ be in $U(D, q)$ and define the Nemytski
operators
\begin{equation}\label{nemitski}
\tilde{\sigma}(t,u)(x)=\tilde{g}(t,x,u(x)),\ \ {\rm and }\ \
\sigma(t,u)\psi(x)=\sum_{1\leq i\leq m} g_{i}(t,x,u(x))\psi_{i}(x),
\end{equation}
where $\psi_i \in H_{0}$, $i=1, \cdots, m$.  These operators satisfy the assumptions {\bf
(C3q)} and {\bf (C3qBis)} (see e.g. \cite{BP2001}).
The condition $U(D,\infty)$ obviously implies $U(D,q)$ for every $q\in [2,\infty)$.
Therefore, if the coefficients $\tilde{g}$ and $g_i$ belong to the class $U(D,\infty)$,
then $\s$ and $\tilde{\s}$  satisfy
the conditions {\bf (C1)}, {\bf (C1Bis)}, {\bf (C2)}, {\bf (C2Bis)},
  {\bf (C3q)} and {\bf (C3qBis)} for all $q\in [2,\infty)$.

\subsection{The Skorohod-Jakubowski representation theorem}\label{Jakubowski}
Let $\mathcal{Z}$ be a topological space such that there exists a sequence $(f_{j})$ of continuous functions 
$f_j : \mathcal{Z} \to [-1,1]$ 
that separate points of $\mathcal{Z}$. 

The following result is proved in \cite{Jaku97}, Theorem 2. 
\begin{theorem}
Let $(\mathcal{P}_{j}, j\in {\mathbb N})$  be a tight sequence of Borel probability measures on $\mathcal{Z}$.
Then there exist a subsequence $(j_{k})$ and Borel measurable maps
$ \theta_{k}: [0,1]\rightarrow\mathcal{Z},\quad k\geq 1$ such that 
for each $k\geq 1$, $\mathcal{P}_{j_k}$ is equal to the law of
$\theta_{k}$ and for every $s\in [0,1],\quad  \theta_{k}(s)\rightarrow \theta(s)\quad {\rm in}\quad
\mathcal{Z}$ as $k\to \infty$.  
\end{theorem}

\smallskip

\textbf{Acknowledgments:} The work of H. Bessaih was partially
supported by the NSF grant No. DMS 0608494.
This work was partially written while the authors were staying at the Isaac Newton Institute
for Mathematical Sciences in Cambridge. They would like to thank the Institute for the financial
support, the very good working conditions and the friendly atmosphere.

\end{document}